\documentclass[a4paper,12pt]{amsart}
\usepackage[utf8]{inputenc}
\usepackage{amssymb,amsmath}
\usepackage{mathbbol}
\usepackage{amscd}
\usepackage{calrsfs}
\usepackage{cancel}
\usepackage[all]{xy}
\usepackage[onehalfspacing]{setspace}
\usepackage{enumerate}
\usepackage{graphicx}
\usepackage{url}
\usepackage[a4paper, left=2cm, right=2cm, top=3.5cm, bottom=3cm]{geometry}
\usepackage{tikz-cd}
\usepackage{array}
\usepackage[normalem]{ulem}
\usepackage{verbatim}
\usepackage[stretch=10]{microtype}
\usepackage[colorlinks, linktocpage, citecolor = blue, linkcolor = blue,backref]{hyperref} 
\usepackage{braket}

\def\A{\mathbb A}
\def\C{\mathbb C}
\def\F{\mathbb F}

\def\k{\mathbb k}
\def\L{\mathbb L}

\def\P{\mathbb P}
\def\bP{\mathbb P}

\def\Q{\mathbb Q}

\def\Z{\mathbb Z}

\def\OO{\mathcal O}

\def\XX{\mathcal X}

\def\chr{{\operatorname{char}}}

\def\Ker{{\operatorname{Ker}}}

\def\Image{{\operatorname{Im}}}

\def\Spec{{\operatorname{Spec}}}

\def\ord{{\operatorname{ord}}}
\def\Pic{{\operatorname{Pic}}}
\def\NS{{\operatorname{NS}}}

\def\Bl{{\operatorname{Bl}}}

\def\Gr{{\operatorname{Gr}}}

\def\PGL{{\operatorname{PGL}}}

\def\codim{{\operatorname{codim}}}

\def\Aut{{\operatorname{Aut}}}

\def\rk{{\operatorname{rk}}}

\def\chr{{\operatorname{char}}}
\def\Bir{{\operatorname{Bir}}}
\def\StBir{{\operatorname{StBir}}}
\def\Cr{{\operatorname{Cr}}}
\def\Rep{{\operatorname{Rep}}}

\def\Ex{{\operatorname{Ex}}}

\def\Gal{{\operatorname{Gal}}}
\def\Jac{{\operatorname{J}}}

\def\tc{\widetilde{c}}

\def\ol{\overline}

\def\Kt{{\mathrm{K}}}

\def\Id{{\mathrm{Id}}}

\def\OG{{\mathrm{OG}}}
\def\PPAV{{\mathrm{PPAV}}}

\def\Var{{\mathrm{Var}}}

\def\ab{{\mathrm{ab}}}

\theoremstyle{plain}
\newtheorem{dummy}{dummy}[section]
\newtheorem{theorem}[dummy]{Theorem}
\newtheorem{proposition}[dummy]{Proposition}

\newtheorem{lemma}[dummy]{Lemma}
\newtheorem{lem}[dummy]{Lemma}

\newtheorem{corollary}[dummy]{Corollary}
\newtheorem{cor}[dummy]{Corollary}

\newtheorem{example}[dummy]{Example}

\newtheorem{definition}[dummy]{Definition}
\numberwithin{equation}{section}
\theoremstyle{definition}
\newtheorem{Def}[dummy]{Definition}
\newtheorem{remark}[dummy]{Remark}

\newcommand{\ssec}{\subsection}

\newcommand{\ti}[1]{\tilde{#1}}
\newcommand{\ul}{\underline}
\newcommand{\vast}{\bBigg@{4}}
\newcommand{\Vast}{\bBigg@{5}}

\newcommand{\cD}{\mathcal{D}}

\newcommand{\cO}{\mathcal{O}}

\newcommand{\cU}{\mathcal{U}}

\newcommand{\cX}{\mathcal{X}}

\newcommand{\gS}{\Sigma}

\newcommand{\gb}{\beta}

\newcommand{\go}{\omega}
\newcommand{\gs}{\sigma}

\newcommand{\Frac}{\mathrm{Frac}}

\newcommand{\psreg}{\mathrm{p-reg}}

\newcommand{\sep}{\mathrm{sep}}

\newcommand{\Sym}{\mathrm{Sym}}

\newcommand{\X}{\mathrm{X}}
\newcommand{\bss}{\backslash}

\newcommand{\cnec}{\mathrel{:=}}
\newcommand{\dto}{\dashrightarrow}

\newcommand{\hto}{\hookrightarrow}
\newcommand{\hlto}{\hookleftarrow}
\newcommand{\xto}[1]{\xrightarrow{ #1 }}

\title{Motivic invariants of birational maps}

\author{Hsueh-Yung Lin, Evgeny Shinder}

\address{Department of Mathematics, National Taiwan University, 
	No. 1, Sec. 4, Roosevelt Rd., Taipei 10617, Taiwan.}
\email{hsuehyunglin@ntu.edu.tw}

\address{School of Mathematics and Statistics, University of Sheffield,
Hounsfield Road, S3 7RH, UK, and
Hausdorff Center for Mathematics
at the University of Bonn, Endenicher Allee 60, 53115.}
\email{eugene.shinder@gmail.com}


\begin{document}

\raggedbottom

\maketitle

\begin{abstract}
We construct invariants of birational maps 
with values in the Kontsevich--Tschinkel group
and in the truncated Grothendieck groups of varieties.
These invariants are morphisms of groupoids and are
well-suited 
to investigating the structure
of the Grothendieck ring and L-equivalence.
Building on
known constructions
of L-equivalence, we prove new unexpected
results about
Cremona groups.
\end{abstract}

\setcounter{tocdepth}{1}
\tableofcontents 

\section{Introduction}

Let $\k$ be a field.
The Cremona groups $\Cr_n(\k) = \Bir(\P^n_{\k})$
have been actively studied in birational geometry since
the 19th century.
The aim of this work is to contribute to the study of Cremona groups
through a motivic viewpoint 
and bring new methods and results to light.

\ssec{Generating sets of Cremona groups}
\hfill

We address the problem
about the existence of simple generating sets of Cremona groups.
For the Cremona group of complex projective plane,
it has a well-known generating set due to
M. Noether and Castelnuovo.

\begin{theorem}[Noether--Castelnuovo~{\cite{zbMATH02722071,CastelnuovoCr2}}]\label{thm:noether}
For any algebraically closed field $\k$,
the Cremona group $\Cr_2(\k)$ is generated by $\PGL_3(\k)$
and the Cremona involution 
\[[X:Y:Z] \mapsto [YZ:XZ:XY].\]
\end{theorem}

When $n \ge 3$, the Cremona groups $\Cr_n(\k)$ 
are much less well understood. 
In terms of the size of generating sets, Hudson and Pan showed that 
for an algebraically closed field $\k$ of characteristic zero,
$\Cr_n(\k)$ is never generated by $\PGL_{n+1}(\k)$
together with any set of transformations of bounded degree 
or countably many elements
when $\k$ is uncountable~\cite{HudsonPan}.
It is therefore an intricate
question in higher dimension whether there exist,
group-theoretically or geometrically,
simple transformations
which generate the entire $\Cr_n(\k)$.
See e.g.~\cite[1.C]{BLZ} for a historical account of this problem,
which has been already mentioned in the 1895 
lectures by Enriques.
Currently, no explicit set of generators
of Cremona groups in dimension $n \ge 3$ is known.

Among the birational automorphisms,
regularizable maps form 
one of the most explicit classes
of elements of $\Cr_n(\k)$.
Recall that an element $\phi \in \Cr_n(\k)$
is called regularizable
if there exists a birational map 
$\alpha: \P^n \dashrightarrow X$
to some variety $X$ 
such that $\alpha \circ \phi \circ \alpha^{-1} \in \Bir(X)$ is a regular automorphism.
See Definition~\ref{Def-reg} for 
a more general definition of 
pseudo-regularizable maps. 
Regular automorphisms of $\P^n$ are regularizable,
and so are elements of $\Aut(U) \subset \Cr_n(k)$
for every open $U \subset \P^n$.
Furthermore, all
elements of $\Cr_n(\k)$ of
finite order are regularizable
by Weil's regularization theorem (see e.g.~\cite[Theorem 1.7]{CheltsovCremona}).
Thus Theorem \ref{thm:noether} implies that $\Cr_2(\C)$
is generated by regularizable elements.

In higher dimensions, our main result shows that
the contrary holds in most situations and accordingly, 
any set of generators
of $\Cr_n(\k)$
is necessarily quite complicated.

\begin{theorem}\label{thm:main}
In each of
the following cases,
the Cremona group
$\Cr_n(\k)$ is not generated by pseudo-regularizable elements
\textup(which include $\PGL_{n+1}(\k)$
and all 
elements of finite order\textup):
\begin{enumerate}
    \item $n = 3$, and all number fields $\k$, or all function fields $\k$ 
    over a number field, over a finite field or over an algebraically closed field; 
    \item $n \ge 4$, and all fields $\k \subset \C$;
    \item $n \ge 5$, and all infinite fields $\k$.
\end{enumerate}
\end{theorem}

We refer to Theorem \ref{thm:bir-generation} for a more general statement regarding $\Bir(X)$
for classes of varieties $X$ birational to 
$\P^n \times W$ where $W$ is a (separably)
rationally connected variety.
Theorem \ref{thm:main}
disproves a conjecture by Cheltsov in 2004,
which says that 
birational automorphism groups are generated by regularizable
elements \cite[Conjecture 1.12]{CheltsovCremona}.
It also
gives a negative answer to
a more specific
question of Dolgachev \cite[Question on p.1]{Deserti} on whether Cremona groups
are generated by involutions.
We will explain below why certain
types of fields appear in Theorem~\ref{thm:main}.
For now let us note that even though
any field extension $\k \subset \k'$
induces
an embedding $\Cr_n(\k) \subset \Cr_n(\k')$,
there is no straightforward group-theoretic comparison
between $\Cr_n(\k)$ and 
$\Cr_n(\k')$ (in particular $\Cr_n(\ol{\k})$) in
terms of their generating sets.

\ssec{Motivic invariants}
\hfill

Our method of proof relies on a 
motivic 
invariant of birational maps.
Here, by `motivic' we mean
in the sense of
the Grothendieck ring of varieties,
which is one the original
constructions 
of the ring of motives due to Grothendieck, 
and the Burnside ring of Kontsevich--Tschinkel
\cite{KontsevichTschinkel}.
For an $n$-dimensional variety $X$ over $\k$,
this is a group homomorphism
\[
c: \Bir(X) \to \Z[\Bir_{n-1}/\k].
\]
Here $\Bir_{n-1}/\k$ is the set of $(n-1)$-dimensional
birational
equivalence classes of varieties over $\k$.
The invariant $c$ records the birational
classes of 
the exceptional divisors created
and contracted by the birational map.
We prove that in the situations described
in Theorem \ref{thm:main} for $X = \P^n$, 
the invariant 
is non-trivial, 
while it vanishes on all pseudo-regularizable elements.
Furthermore, 
using the explicit non-vanishing of the
invariant $c$ we can 
construct infinitely many
homomorphisms $\Cr_n(\k) \to \Z$ for various
fields $\k$ and $n \ge 3$, see Corollary~\ref{cor:homZ}.
This is in contrast to the 
homomorphisms $\Cr_n(\k) \to \Z/2$ constructed in~\cite{BLZ}.

Our invariant $c$ 
has an enhancement
\[
\ti{c}: \Bir(X) \to \Kt_0(\Var^{\le n-1}/\k),
\]
taking values in the truncated
Grothendieck group of varieties
in such a way that $c$ is the composition
\[
\Bir(X) \overset{\ti{c}}{\to} \Kt_0(\Var^{\le n-1}/\k) \to \Z[\Bir_{n-1}/\k],
\]
where the second homomorphism is the projection to the top dimensional
birational classes. We develop the invariant $\ti{c}$ 
because some proofs are more transparent in this context,
and also because
it provides full control over
the Grothendieck ring of varieties, 
as it determines how
$\Kt_0(\Var^{\le n-1}/\k)$ is mapped to $\Kt_0(\Var^{\le n}/\k)$ 
(see Proposition~\ref{prop:ker-iota}).
None of these constructions require
any form of the
resolution of singularities, and
we work over an arbitrary field
$\k$.

Once the invariants $\ti{c}$ and $c$ are constructed
and their properties settled in Section \ref{sec:invariants}, 
in Section \ref{sec:geometric-constructions}
we introduce and
construct examples of the so-called L-links 
(L stands for the Lefschetz class $\L$
in the Grothendieck ring and L-equivalence \cite{KuznetsovShinder}); 
these links provide a geometric input suitable
for L-equivalence constructions 
and for showing the non-vanishing of $c(\Bir(X))$.
To construct non-trivial L-links, 
we rely on the geometric constructions
of L-equivalence
involving genus one curves
\cite{Mori-Mukai, Crauder-Katz, ShinderZhang}, 
K3 surfaces \cite{HassettLai}, 
and Calabi-Yau threefolds \cite{IMOUG2}.
Even though these constructions already exist in the literature when
$\k = \C$,
much effort goes into generalizing them and checking 
that such L-links $\phi$ are motivically non-trivial,
namely $c(\phi) \ne 0$,
over a more general field $\k$ as in Theorem \ref{thm:main}.

For instance, links involving genus one curves 
become motivically trivial over $\C$, 
but not over
small fields as 
in Theorem \ref{thm:main}(1);
roughly speaking, 
Theorem \ref{thm:main}(1) 
works for these fields
because they have infinitely many
Galois extensions of fixed degree.
To prove Theorem~\ref{thm:main}(2),
we use the construction~\cite{HassettLai},
taking as input a general complex
K3 surface of degree $12$ and
Picard rank one, 
together with some rational points on it.
We extend the construction to $\Q$ nevertheless,
representing K3 surfaces
of degree $12$ 
as hyperplane sections of
a \emph{rational} Fano threefold of degree $12$
defined over $\Q$,
and using Terasoma's argument to prove the existence of K3 surfaces
over $\Q$ of Picard rank one~\cite{Terasoma}
involving Lefschetz pencils.
Finally to prove Theorem \ref{thm:main}(3),
we need to show that the L-links with Calabi-Yau threefolds as centers \cite{IMOUG2}
exist and
can be nontrivial
over any infinite field $\k$.
We have to deal with 
Calabi-Yau varieties 
which are possibly uniruled, and
standard arguments in characteristic zero involving MRC fibrations
do not apply. Instead, we rely on
a weaker statement in the spirit of MRC fibrations
for separably rationally connected varieties,
which we prove in Appendix~\ref{app:MRC}.
We also include
Theorem~\ref{thm:BurtCY},
due to Burt Totaro, showing that
birational
varieties with nef canonical class and 
Picard number
one are isomorphic, in arbitrary
characteristic, which was previously
known in characteristic zero or in positive characteristic in dimension up to three.

\subsection{Relation to other work}

\subsubsection{L-equivalence} 
\hfill

This paper was motivated by 
understanding the geometric meaning
of the information encoded by the Grothendieck ring of varieties and
its variants,
which have played a spectacular role in the recent
results on the specialization of stable rationality \cite{NicaiseShinder} and rationality \cite{KontsevichTschinkel}.
In particular 
we want to understand
L-equivalence \cite{Borisov}, \cite{HassettLai},
\cite{KuznetsovShinder}, \cite{IMOUG2}; our concept of L-links
formalizes the existing geometric constructions leading to L-equivalence.

This is a continuation of our work \cite{LSZ20} where
jointly with Susanna Zimmermann we have studied
the invariant $c(\phi)$ for surfaces $X$ over perfect fields 
using Minimal Model Program.
We showed that
$c(\Bir(X)) = 0$ 
and as a consequence, birational maps between surfaces 
do not produce non-trivial L-equivalence.

On the other hand, 
non-trivial L-equivalence
already exists from dimension $3$ and on,
as we will
explain a construction in Subsection \ref{ssec-dim3} originating from curves in rational threefolds. 
Namely, by Theorem \ref{thm:ell-curves}, we have $\L([C] - [C']) = 0$ for a pair of genus one curves.
This improves the exponent of $\L$ in 
\cite{ShinderZhang} from $\L^4$ to the minimal possible.

\subsubsection{Studying birational maps through the blow up centers} 
\hfill

Making an invariant $c$ from the blow up centers
or the exceptional divisors is similar in spirit 
to the filtration on the Cremona groups by the 
genus of the curves blown up by 
the maps
of threefolds,
considered and studied
by Frumkin \cite{Frumkin-Cremona},
Lamy
\cite{Lamy-Cremona}, Bernardara \cite{Bernardara-Cremona}
and Blanc--Cheltsov--Duncan--Prokhorov \cite{BCDP}.
In fact similar considerations also appear in 
the proof of the Hudson--Pan theorem
\cite{HudsonPan}, which relies on 
the existence of a large
set of possible types of exceptional divisors
for Cremona transformations.
Comparing to the above works,
our approach has the key feature
that the invariant $c$ is a \emph{homomorphism}, which in a way
refines these filtrations.

The relationship between birational maps and the structure
of the Grothendieck ring of varieties also
appears in the work
of Zakharevich in terms of the differentials
in a certain spectral sequence \cite{Zakharevich}
which converges to the higher K-theory groups
lying over
the Grothendieck ring of varieties.
Our invariant $c(\phi)$ coincides with the first
differential from that spectral sequence \cite[Lemma 3.2]{Zakharevich},
and the invariant $\ti{c}(\phi)$ is a certain
enhancement of that map.

On the other hand, relating the Grothendieck
ring of varieties and the Cremona groups is
implicitly suggested in \cite{HassettLai}.
Part of the motivation for our work was to understand
to which extent the blow up centers of a rationality construction $\P^n \dto X$ are determined
by $X$ itself, which was asked in \cite[Introduction]{HassettLai}.
The intrinsic ambiguity in these centers 
leads to the non-vanishing of $c(\Bir(X))$,
which makes our applications to the Cremona groups possible.
Along the way we answer a question
of Hassett and Tschinkel about the 
realization of non-isomorphic genus one curves as
 factorization centers 
for rational threefolds over non-closed fields~\cite[Remark 23 (2)]{HassettTschinkelCycle}, by Theorem \ref{thm:ell-curves} which
was already mentioned above.

\subsubsection{Non-simplicity of the Cremona groups and constructions of non-trivial homomorphisms} 
\hfill

There has been important progress concerning the non-simplicity of the Cremona groups
over the last decade;
see~\cite{Cantat-Cremona} and~\cite[Introduction]{BLZ}
for an excellent overview of the history and recent results.
First of all,
in 2012
Cantat and Lamy proved that 
for any algebraically closed
field $\k$,
$\Cr_2(\k)$ is not a simple group~\cite{MR3037611}.
This result was later extended to $\Cr_2(\k)$ for any field $\k$ by Lonjou~\cite{MR3533276}.
Using the Minimal Model Program and Birkar's boundedness,
Blanc--Lamy--Zimmermann~\cite{BLZ} recently constructed infinitely
many homomorphisms $\Cr_n(\k) \to \Z/2$ ($n \ge 3$) when $\k$ is a subfield of $\C$ (see also~\cite{MR4136435}). 
Their constructions
originate from Sarkisov links and relations between them, and
are of quite different nature to 
the homomorphisms 
$\Cr_n(\k) \to \Z$ we obtain 
in Corollary~\ref{cor:homZ}
through the motivic invariants.

\subsubsection{Generation of Cremona groups
by involutions}
\hfill

As a consequence of Theorem \ref{thm:noether},
$\Cr_2(\C)$ is generated by involutions.
The question whether same is true in higher dimension
was asked by Dolgachev during his series of lectures
in Toulouse in 2016.
Motivated by this question, D\'eserti has proved that $\PGL_{n+1}(\C)$
is generated by involutions \cite[Proposition C]{Deserti}.
Lamy and Schneider have proved
that Cremona groups $\Cr_2(\k)$ over a perfect field $\k$ are generated
by involutions \cite{LamySchneider}, thus providing
a positive answer to Dolgachev's question for $n = 2$.
On the other hand,
Blanc, Schneider and Yasinsky announced 
that birational automorphisms of Severi--Brauer surfaces
are not generated by elements of finite order \cite{BSY},
extending previous results by Shramov \cite{Shramov-SB}.
As a consequence of their result, 
they
are able to give a different proof that $\Cr_n(\C)$
for $n \ge 4$
is not generated by $\PGL_{n+1}(\C)$ and elements of finite order.

\subsection{Notation and conventions}
\label{ss:notation}
\hfill

Unless specified otherwise, we work over an arbitrary field $\k$.
By a variety over $\k$,
we mean
an 
irreducible
and reduced separated
scheme
of finite type over $\k$.
We will use various Grothendieck groups and rings, which we summarize for the convenience of the reader:
\begin{enumerate}
\item[(a)] $\Z[\Bir/\k]$ (resp. $\Z[\Bir_n/\k]$) 
is the free abelian group generated by birational isomorphism classes of varieties 
(resp. varieties of dimension $n$)
over $\k$.
By definition 
$\Z[\Bir/\k] = \bigoplus_{n \ge 0}
\Z[\Bir_n/\k]$. This is the graded Burnside ring of Kontsevich--Tschinkel \cite{KontsevichTschinkel};
\item[(b)] $\Kt_0(\Var/\k)$ 
(resp. $\Kt_0(\Var^{\le n}/\k)$)
is the Grothendieck ring (resp. group)
of varieties (resp. varieties of dimension $\le n$)
over $\k$
modulo cut and paste relations; 
see Subsection~\ref{ssec-tG0} for the precise definitions;
\item[(c)] $\Kt_0(\PPAV/\k)$ is the Grothendieck group of the additive category of 
principally polarized abelian varieties over $\k$
modulo relations $[A \times B] = [A] + [B]$;
\item[(d)] for a profinite group
$G$, $\Kt_0(\Rep(G, \Q_\ell))$
is the Grothendieck group of
the abelian category of
finite-dimensional
continuous $G$-representations 
with coefficients in $\Q_\ell$. See \cite[Section 2.1]{LSZ20} for the details of this construction.
\end{enumerate}
Among these,  (a) and (b)
are central for this paper and
are explained in detail and 
related to each other in Subsection \ref{ssec-tG0}.
The Grothendieck groups (c) and (d)
will be used for some technical purposes
as ``motivic realizations'' of varieties and birational classes,
see in particular \eqref{eq:sigma}, \eqref{eq:j}.

\section{The motivic invariants}
\label{sec:invariants}

\ssec{Invariant $c(\phi)$ with values in $\Z[\Bir_{n-1}/\k]$}
\hfill

Let $\k$ be a field and
let $X$ and $Y$ be two varieties over $\k$.
A birational map $\phi : X \dto Y$
is a morphism $U \to Y$ over $\k$
defined on a Zariski dense open subscheme
$U \subset X$
which is isomorphic onto its image.
Two birational maps $\phi,\psi : X \dto Y$
are identical
if there exists an open 
$U \subset X$ such that $\phi_{|U} = \psi_{|U}$
as morphisms.
We can compose birational maps and each
birational map $\phi$ has an inverse $\phi^{-1}$,
thereby defining a
groupoid $\ul{\Bir}/\k$ of birational classes
consisting of $\k$-varieties as objects
and birational maps as morphisms.

We say that $X$ and $Y$ are birational if there is a birational
map between them. We say that $X$ and $Y$ are stably
birational if $X \times \P_\k^k$ and $Y \times \P_\k^\ell$ are birational
for some $k, \ell \ge 0$.
The set of birational equivalence classes of 
$n$-dimensional
$\k$-varieties
is denoted by $\Bir_n/\k$.
The set of birational maps 
from $X$ to $Y$ is denoted by $\Bir(X,Y)$.

Let $\phi: X \dto Y$ be a birational map between
$n$-dimensional $\k$-varieties.
We say that $\phi$ is an isomorphism at a point 
$x \in X$ if there exists a Zariski open $U \subset X$
containing $x$ such that
$\phi_{|U}$ is an open embedding into $Y$.
Here, the point $x \in X$ is not required to be closed.
We define the exceptional set as
$$\Ex(\phi) \cnec \Set{x \in X | \phi \text{ is not an isomorphism at } x} \subset X.$$
According to the definition, $\Ex(\phi)$ is a closed
subset
containing the indeterminacy locus of $\phi$.

By definition, $\phi$ establishes
an isomorphism
\begin{equation}\label{eq:iso-ex}
X \setminus \Ex(\phi) \simeq Y \setminus \Ex(\phi^{-1})
\end{equation}

\begin{Def}
For every birational map $\phi: X \dto Y$ of $\k$-varieties,
we set
\begin{equation}\label{eq:def-c}
c(\phi) \ \ \cnec 
\sum_{\substack{y \in \Ex(\phi^{-1}) \\ \codim_Y y = 1}} [\k(y)]
 \ \ -  \sum_{\substack{x \in \Ex(\phi) \\ \codim_X x = 1}} [\k(x)] \ \ \in \ \ \Z[\Bir_{n-1}/\k]
 \end{equation}
where $\Z[\Bir_{n-1}/\k]$ 
is the free abelian group generated by $\Bir_{n-1}/\k$,
and where the function fields $\k(y)$, $\k(x)$ are identified
with the corresponding birational classes.
\end{Def}

Since both $\Ex(\phi) \subset X$ 
and $\Ex(\phi^{-1}) \subset Y$ are Zariski closed,
the sums in the definition of $c(\phi)$ are finite.
The notation $c$ stands for \emph{centers}
of the exceptional divisors \cite[Definition 2.24]{KollarMori}, 
as they are
encoded in~\eqref{eq:def-c} 
in a certain way (see e.g. Proposition~\ref{pro-cWF}).

A crucial property
of $c$ is that it defines a homomorphism
from the groupoid $\ul{\Bir}/\k$
of birational classes to
$\Z[\Bir/\k]$.

\begin{lem}\label{lem-homc}
Given birational maps $\phi : X \dto Y$ and $\psi : Y \dto Z$,
we have
$$c(\psi \circ \phi) = c(\phi) + c(\psi).$$
\end{lem}

Lemma~\ref{lem-homc} has a natural proof
based on motivic cut-and-paste,
and we postpone the proof until Subsection 
\ref{ssec-leftK},
see Theorem \ref{thm-Defc}.

Note that we work over arbitrary fields.
In particular,
the definition of $c$ and the proof of Lemma~\ref{lem-homc}
do not rely on 
any version of the resolution of singularities.
Over fields of characteristic zero 
or for surfaces over arbitrary fields
we understand birational maps better, 
thanks to
Hironaka's resolution of singularities (see e.g.~\cite{KollarRes}) 
and
Weak Factorization~\cite{WF},~\cite[Lemma 54.17.2]{stacks-project}.
The latter asserts that
every birational map $\phi : X \dto Y$
between smooth complete varieties over $\k$
admits a factorization
\begin{equation}\label{eq:wf}
\xymatrix{
 & Y_1 \ar[dr]^{q_1} \ar[dl]_{p_1} & & Y_2 \ar[dl]_{p_2} & \cdots & Y_{m-1} \ar[dr]^{q_{m-1}} & & Y_m  \ar[dr]^{q_m} \ar[dl]_{p_m}  \\
X & & X_1 & & \cdots & & X_{m-1} & & Y \\
}
\end{equation}
where each $p_i$ (resp. $q_i$) is a sequence of blow ups
(resp. blow downs)
along smooth irreducible closed subvarieties
of codimension $\ge 2$ or an isomorphism
(which we understand
as a blow up with empty center).

\begin{proposition}\label{pro-cWF}
If a birational map $\phi: X \dto Y$ between $n$-dimensional
smooth complete $\k$-varieties 
admits
a factorization \eqref{eq:wf},
then 
\begin{equation}\label{eqn-cWF}
c(\phi) = 
\sum_{Z \in \cU} [\P^{n - \dim(Z) - 1} \times Z]
  -  \sum_{T \in \cD} [\P^{n - \dim(T) - 1} \times T]  \in  \Z[\Bir_{n-1}/\k]
\end{equation}
where $\cU$ \textup(resp. $\cD$\textup)
is the set of blow up centers 
of $p_i$ \textup(resp. $q_i$\textup).

In particular, the alternating sum in~\eqref{eqn-cWF} is well-defined, namely it is independent of the
choice of weak factorization of $\phi$.

\end{proposition}

In \eqref{eqn-cWF} we use the convention that $[\emptyset] = 0$.

\begin{proof}

Proposition~\ref{pro-cWF} follows easily from
Lemma \ref{lem-homc} and the definition of $c$,
because the exceptional divisors of the $p_i$'s are of the form
$\P^{n-\dim(Z)-1} \times Z$, for $Z \in \cU$
and similarly for the $q_i$'s.
\end{proof}

We can relate $c(\phi)$ to the Galois action
on the exceptional
divisors.
Let $G_\k = \Gal(\k^{\sep}/\k)$
and let $\Kt_0(\Rep(G_\k,\Q_\ell))$ be the Grothendieck
ring of continuous finite-dimensional 
$\ell$-adic representations of $G_\k$; here
$\ell$ is a fixed prime number different from $\chr(\k)$. 
Define the group homomorphism
\begin{equation}\label{eq:sigma}
\gs: \Z[\Bir/\k] \to \Kt_0(\Rep(G_\k,\Q_\ell))
\end{equation}
by sending the class of a $\k$-irreducible variety 
$X$ to the permutation representation
on the irreducible components of $X_{\k^{\sep}}$
with coefficients in $\Q_\ell$.

For a smooth projective variety $X/\k$, 
we consider the N\'eron-Severi group with $\Q_\ell$-coefficients
$N(X) := \NS(X_{\k^{\sep}}) \otimes \Q_\ell$.
Then $N(X)$ is a finite-dimensional $\Q_\ell$-vector
space \cite[Theorem II.4.5]{Kollarrat} with a
continuous 
Galois action.

\begin{proposition}\label{prop:Picnb}
Let $\k$ be a field 
of characteristic zero.
For every birational map $\phi : X \dto Y$ 
between smooth projective 
varieties,
we have
$$\gs(c(\phi)) = [N(Y)] - [N(X)].$$
In particular, if $\phi \in \Bir(X)$,
then $\sigma(c(\phi)) = 0$.
\end{proposition}
\begin{proof}

By the Weak Factorization Theorem, 
the statement is reduced to the case where
$\phi: X \to Y$ is a single blow up 
with a smooth
$\k$-connected center $Z$ of codimension $\ge 2$.
In this case
the result follows
from the isomorphism of Galois representations
$N(X) \simeq N(Y) \oplus \Q_\ell[E]$,
where $\Q_\ell[E]$ stands for the permutation
representation on the geometric irreducible components
of $E$, because $\sigma(c(\phi)) = \sigma(-[E]) = -[\Q_\ell[E]]$.
\end{proof}

In the rest of this section, we 
first explain for which kinds of varieties
$c$ is identically zero,
then enhance $c(\phi)$ to 
an invariant $\ti{c}(\phi)$
which takes values in the
Grothendieck groups of varieties truncated
by the dimension.
We will explain how these invariants control the structure
of the Grothendieck ring
viewed as the 
colimit of its truncations.
The reader who is interested in examples and geometric applications
can skip these parts and go directly
to Section~\ref{sec:geometric-constructions}.

\subsection{Vanishing results for $c(\phi)$}
\hfill

Given a birational automorphism 
$\phi: X \dto X$ of variety $X$,
there are situations where $c(\phi) = 0$ is automatic.
First of all, by definition
this is always the case when $\phi$
is an isomorphism, or just a pseudo-isomorphism (isomorphism in codimension one).
This applies 
to smooth projective curves
and to all
smooth projective varieties $X$
with $K_X$ is nef (e.g. Calabi--Yau varieties),
by
Theorem~\ref{thm:BurtCY}.

In dimension two,
the following vanishing result was proven in~\cite{LSZ20}.

\begin{theorem}\cite[Theorem 3.4]{LSZ20}
\label{thm:surfaces-main}
If $\k$ is a perfect field and
$X/\k$ is a surface,
then $c(\Bir(X)) = 0$.
\end{theorem}

In dimension $3$,
the vanishing $c(\Bir(X)) = 0$ still holds in certain cases,
mostly depending on the base field $\k$.
The proof relies on the existence
of the principally
polarized intermediate Jacobian defined for
rationally connected threefolds over non-algebraically
closed
fields \cite{AchterCMVial}, \cite{AchterCMVial-2}, \cite{BenoistWittenbergPerf}.

\begin{proposition}\label{prop:ratconn3-fold}
Let $\k$ be a field of characteristic zero,
and $X$ be a smooth projective
rationally connected threefold.
The image $c(\Bir(X))$ is contained
in the subgroup of $\Z[\Bir_{2}/\k]$ generated
by classes of the form 
$[\P^1 \times C]$ where $C$ 
is an irreducible smooth curve
whose geometric
irreducible components have genus zero or one.
If $\k$ is algebraically closed, then 
$c(\Bir(X)) = 0$.
\end{proposition}

See Theorem \ref{thm:ell-curves} for genus one
curves appearing in the image of $c(\Bir(\P^3))$
over non-algebraically closed fields.

\begin{proof}
By the
Weak Factorization theorem
and Proposition~\ref{pro-cWF},
the image $c(\Bir(X))$ is contained
in the subgroup of $\Z[\Bir_2/\k]$
generated by classes $[\P^1 \times C]$
for smooth irreducible curves $C$
(if the blow up center is a closed point $Z$, the
corresponding term is $[\P^2 \times Z]
= [\P^1 \times \P^1_Z])$.

Now consider the assignment
\begin{equation}\label{eq:j}
[X] \mapsto \left\{\begin{array}{cl}
     \; \Jac(C), & \text{$X$ is birational to a 
     ruled surface over $C$} \\
     \; 0, & \text{otherwise} \\
\end{array}\right.
\end{equation}
for every birational class $[X] \in \Bir_2/\k$ of irreducible surfaces, where
$C$ is taken to be an irreducible smooth projective curve,
and $\Jac(C)$ is the Jacobian of $C$.
This defines a
group homomorphism 
$$j: \Z[\Bir_2/\k] \to 
\Kt_0(\PPAV/\k),$$
where $\Kt_0(\PPAV/\k)$ is the Grothendieck 
group of principally
polarized abelian varieties.

Let us prove that
for any $\phi \in \Bir(X,Y)$
where $X$ and $Y$ are
smooth projective rationally connected threefolds, we have
\begin{equation}\label{eq:j-c-jac}
j(c(\phi)) = [\Jac(Y)] - [\Jac(X)],
\end{equation}
where $\Jac(X)$ is the intermediate Jacobian of $X$~\cite{BenoistWittenbergPerf}.
Once again by Weak Factorization
and Lemma~\ref{lem-homc},
it suffices to check~\eqref{eq:j-c-jac}
for a single blow up along an irreducible center. 
This follows from
the blow up formulas
\cite[Lemma 2.10]{BenoistWittenbergPerf}
that $\Jac(\Bl_C(X)) \simeq \Jac(X) \times \Jac(C)$ 
when $C \subset X$ is a smooth projective curve,
and $\Jac(\Bl_P(X)) \simeq \Jac(X)$ for a point $P \in X$.
In particular, for any $\phi \in \Bir(X)$ we have
$$c(\phi) \in \Ker(j) \cap ([\P^1] \cdot \Z[\Bir_1/\k]).$$

For every irreducible smooth projective curve $C$ over $\k$,
the Jacobian $\Jac(C)$ is indecomposable over $\k$ 
as a principally polarized abelian variety.
Indeed, since $C$ is irreducible,
the $\Gal(\ol{\k}/\k)$-action on the geometric
irreducible components $C_1,\ldots,C_r$,
and therefore on the factors
$\Jac(C_{\ol{\k}}) \simeq \Jac(C_1) \times \cdots \times \Jac(C_r)$
of the decomposition of $\Jac(C_{\ol{\k}})$ is transitive.
Since each $\Jac(C_i)$ is indecomposable,
it follows from the uniqueness of the 
decomposition of principally polarized abelian varieties into irreducible ones
(see \cite[Theorem 3.3]{jordan2016unique} or \cite{DebarrePAV})
that $\Jac(C)$ is indecomposable.

Since principally polarized abelian varieties
admit unique decompositions into indecomposable ones, 
the Grothendieck 
group of principally
polarized abelian varieties
$\Kt_0(\PPAV/\k)$ is isomorphic to the free abelian group
generated by 
indecomposable ones. 
From Serre's Torelli theorem \cite[Appendix]{SerreTorelli},
it follows that $\Jac(C)$, as a principally polarized
abelian variety determines $C$ when all geometric components
of $C$ have genus $\ge 2$.
Thus the kernel
\[
\Ker(\Z[\Bir_1/\k] \overset{[\P^1] \cdot}{\to} \Z[\Bir_2/\k] 
\overset{j}\to \Kt_0(\PPAV/\k))
\]
is contained in the subgroup generated by 
birational classes of
curves with
geometric irreducible components
of genus zero or genus one.
This 
proves the first statement of Proposition~\ref{prop:ratconn3-fold}.

Finally, if $\k$ is algebraically closed,
then the Torelli theorem holds
for genus one curves as well. Thus for any
$\phi \in \Bir(X)$, 
necessarily 
$c(\phi) = m[\P^1 \times \P^1]$ for some $m \in \Z$.
It follows from Proposition \ref{prop:Picnb}
that $m = 0$.
\end{proof}

\subsection{Truncated Grothendieck groups $\Kt_0(\Var^{\le n}/\k)$}\label{ssec-tG0} 
\hfill

Let $\k$ be a field.
Recall that $\Bir_{n}/\k$ (resp. $\Bir/\k$)
denotes the set of birational isomorphism
classes of $n$-dimensional varieties (resp. all varieties),
and we have
the free abelian group 
$$\Z[\Bir/\k] = \bigoplus_{n \ge 0} \Z[\Bir_n/\k].$$
Similarly, let $\Z[\StBir/\k]$ denote 
the 
free abelian group
generated by the stable birational classes.
We have a surjective
homomorphism $\Z[\Bir/\k] \to 
\Z[\StBir/\k]$
sending each birational class to the corresponding
stable birational class.
Furthermore, if
$\chr(\k) = 0$, we have the Larsen-Lunts
isomorphism \cite{LarsenLunts} 
\[\xymatrix{
\Kt_0(\Var/\k) \ar[d] & \Z[\Bir/\k] \ar[d] \\
\Kt_0(\Var/\k) / (\L) \ar[r]_{ \ \ \simeq} &  \Z[\StBir/\k],
}
\]
where $\Kt_0(\Var/\k)$ is the Grothendieck ring of varieties over $\k$ and $\L = [\A^1_\k]$.
There is no natural homomorphism
between
the groups in the upper
row, even in characteristic zero.

We write $\Kt_0(\Var^{\le n}/\k)$ for the Grothendieck
group of varieties of dimension $\le n$: the generators
of $\Kt_0(\Var^{\le n}/\k)$
are
classes $[X]_n$
of schemes of finite type $X/\k$ with $\dim(X) \le n$
and the relations
are generated by
\begin{equation}\label{eq:cut-and-paste}
[X]_n = [U]_n + [Z]_n
\end{equation}
for every open $U \subset X$ with closed complement $Z$.
We have a sequence of natural maps
\begin{equation}\label{eqn-iota}
    \cdots \to \Kt_0(\Var^{\le n-1}/\k) \xto{\iota_{n-1}} \Kt_0(\Var^{\le n}/\k) \xto{\iota_{n}} \Kt_0(\Var^{\le n+1}/\k) \to \cdots.
\end{equation}
The colimit of the direct system defined by~\eqref{eqn-iota}
is the underlying group of 
the Grothendieck ring of varieties
\[
\Kt_0(\Var/\k) = \mathrm{colim}_n \; \Kt_0(\Var^{\le n}/\k).
\]

In this and the next subsections
we will express the kernel and the cokernel of
$$\Kt_0(\Var^{\le n-1}/\k) \xto{\iota_{n-1}} \Kt_0(\Var^{\le n}/\k)$$
for each $n$, 
in terms of information contained in 
the groupoid of birational classes. 
For the cokernel, we have the exact sequence
\begin{equation}\label{eq:K0-seq}
\Kt_0(\Var^{\le n-1}/\k) \overset{\iota_{n-1}}\to \Kt_0(\Var^{\le n}/\k) \overset{\pi_n}\to \Z[\Bir_n/\k] \to 0,
\end{equation}
where 
we define
$\pi_n$ 
on 
the classes of 
finite type schemes
of dimension $\le n$ by
\begin{equation}
\label{eq:def-pi}
\pi_n([X]_n) = \left\{\begin{array}{cc}
     \; [X_1] + \cdots + [X_r], & \dim(X) = n \\
     \; 0, & \dim(X) < n \\
\end{array}\right.,
\end{equation}
where $X_1,\ldots,X_r$ are the irreducible components of $X$ of dimension $n$.
The maps $\iota_{n-1}$ are not injective in general;
in Proposition \ref{prop:ker-iota} we describe their kernels,
see also Remark \ref{rem:iota-injectivity} where we explain
which ones are injective.

The following lemma 
is the starting point
for studying the Grothendieck groups inductively.
See Lemma
\ref{lem:kernelXY} 
for a reformulation of the statement using the motivic invariant
$\ti{c}$.

\begin{lemma}\label{lem:kernelXU}
Let $\Var^n/\k$ denote the set of isomorphism
classes of $n$-dimensional $\k$-varieties.
Then the natural homomorphism 
\begin{equation}\label{eq:plusprojectionU}
\Kt_0(\Var^{\le n-1}/\k) \oplus \Z[\Var^n/\k] \to \Kt_0(\Var^{\le n}/\k)
\end{equation}
is surjective and its kernel
admits the following presentation:
\begin{equation}\label{eq:kernel-XU}
\langle \ \left( [X \setminus U]_{n-1}, 
- [X] + [U] \right) \ \rangle_{U \subset X},
\end{equation}
for all $n$-dimensional
varieties $X$
and open subsets $U \subset X$. 
\end{lemma}

Recall that by our conventions, varieties are assumed
to be irreducible, while
$\Kt_0(\Var^{\le n}/\k)$
is defined with $\k$-schemes of finite type as generators. Lemma
\ref{lem:kernelXU} shows in particular that this distinction
is unimportant, and $\Kt_0(\Var^{\le n}/\k)$
could have been defined with irreducible schemes
of finite type as generators.

\begin{proof}
It is clear from definitions that 
\eqref{eq:plusprojectionU}
is surjective and that
its kernel contains
\eqref{eq:kernel-XU},
so that we have a well-defined surjective homomorphism
\begin{equation}\label{eq:plusprojectionU-new}
\frac{\Kt_0(\Var^{\le n-1}/\k) \oplus \Z[\Var^n/\k]}{\langle \ \left( [X \setminus U]_{n-1}, 
- [X] + [U] \right) \ \rangle_{U \subset X}} 
\overset{\beta}\longrightarrow 
\Kt_0(\Var^{\le n}/\k).
\end{equation}
To show that \eqref{eq:plusprojectionU-new} is also
injective,
we construct the inverse homomorphism $\alpha$ of $\gb$.
First we define $\alpha(X)$ for a scheme $X/\k$
of finite type of dimension $\le n$. 
Since the class $[X]_n$ only depends on the reduced structure,
we can assume that $X$ is reduced.
We stratify $X$ into 
irreducible varieties which are locally closed in $X$ and
let $X_1, \dots, X_r$
be its $n$-dimensional strata
($r = 0$ if $\dim(X) < n$).
The element
\[
\alpha(X) := 
\left([X \setminus \bigsqcup_{i=1}^r X_i ]_{n-1}, \sum_{i=1}^r [X_i]\right) 
\]
is well-defined 
(that is, independent of the choice of stratification of $X$) 
when
considered in the left-hand side of \eqref{eq:plusprojectionU-new}.
We verify that 
$\alpha(X) = \alpha(V) + \alpha(X\setminus V)$
for every open subscheme $V \subset X$, so
that $\alpha$ defines a homomorphism
from $\Kt_0(\Var^{\le n}/\k)$.
By construction $\alpha$ and $\beta$ are mutually inverse.
\end{proof}

\subsection{Invariant $\ti{c}(\phi)$ with values in $\Kt_0(\Var^{\le n-1}/\k)$}\label{ssec-leftK}
\hfill

The next result implies Lemma~\ref{lem-homc}.

\begin{theorem}\label{thm-Defc}
There exists a unique assignment
\[
\ti{c}: \Bir(X,Y) \to \Kt_0(\Var^{\le{n-1}}/\k)
\]
defined on the birational maps 
between $n$-dimensional varieties over $\k$ which
satisfies the following two properties.
\begin{enumerate}
    \item If $i: U \hto X$ is the inclusion of an open subset
    then $\ti{c}(i) = [X \setminus U]_{n-1}$.
    \item For every $\phi \in \Bir(X, Y)$, $\psi \in \Bir(Y,Z)$,
    we have $\ti{c}(\psi \circ \phi) = \ti{c}(\phi) + \ti{c}(\psi)$.
\end{enumerate}
Furthermore, we have $c(\phi) = \pi_{n-1}(\ti{c}(\phi))$,
where $c(\phi)$ is defined by \eqref{eq:def-c} and
$\pi_{n-1}$ is defined by \eqref{eq:def-pi}.
\end{theorem}

\begin{proof}[Proof of Theorem~\ref{thm-Defc} and
Lemma~\ref{lem-homc}]

Let $\phi \in \Bir(X, Y)$, then there is an open subset $i: U \subset X$
such that $j = \phi|_U$ is an open embedding. We have $\phi = j \circ i^{-1}$
so that properties (1) and (2) determine $\ti{c}(\phi)$ to be
\begin{equation}\label{eq:def-ct}
    \ti{c}(\phi) = \ti{c}(j) - \ti{c}(i) = [Y \setminus j(U)]_{n-1} - [X \setminus U]_{n-1}
\end{equation}
This shows that $\ti{c}$ is unique. 

To prove the existence,
we first check that \eqref{eq:def-ct} is well-defined, that is $\ti{c}(\phi)$
is independent of
the choice of $U \subset X$. Let $i': U' \subset X$ be another open subset
such that $j' = \phi|_{U'}$ is an open embedding.
Passing to the intersection of 
$U$ and $U'$ we may assume that $U' \subset U$.
Using the defining relations of $\Kt_0(\Var^{\le n-1}/\k)$
we obtain that
\[
\ti{c}(j') - \ti{c}(i') = [Y \setminus j'(U')]_{n-1} - [X \setminus U']_{n-1} = 
[Y \setminus j(U)]_{n-1} - [X \setminus U]_{n-1} = 
\ti{c}(j) - \ti{c}(i)
\]
hence $\ti{c}$ is well-defined.

To check (2), let $\phi \in \Bir(X, Y)$ and $\psi \in \Bir(Y,Z)$,
and choose $U \subset X$ such that both $j = \phi|_U$ and 
$j'= \psi|_{j(U)}$ are open embeddings. Then
$$\ti{c}(\psi) + \ti{c}(\phi)
= [Z \setminus j'(j(U))]_{n-1} - [Y \setminus j(U)]_{n-1}  +
[Y \setminus j(U)]_{n-1} - [X \setminus U]_{n-1} = \ti{c}(\psi 
\circ \phi).$$

For the final claim, by \eqref{eq:iso-ex} we can factorize $\phi$ through
the inclusions
\[
X \hlto X \setminus \Ex(\phi) \simeq Y \setminus \Ex(\phi^{-1}) \hto Y,
\]
so that by (1) and (2) we have
$\ti{c}(\phi) = [\Ex(\phi^{-1})]_{n-1} - [\Ex(\phi)]_{n-1}$.
It follows from the definitions that $\pi_{n-1}(\ti{c}(\phi)) = c(\phi)$. 
\end{proof}

\subsection{Invariant $\ti{c}$ and
the structure of the Grothendieck rings}
\hfill

The material of this section makes explicit
the first page of
the spectral sequences of
\cite[Section 3]{Zakharevich}
and generalizes \cite[\S 3.2]{LSZ20}.

When we put
$Y = X$ in Theorem 
\ref{thm-Defc},
we obtain a homomorphism
\begin{equation}\label{wtc-BirX}
\ti{c}|_{\Bir(X)} : \Bir(X) \to \Kt_0(\Var^{\le{n-1}}/\k),
\end{equation}
which, as we will see, 
is in general nonzero.
We note that it is crucial that we consider
$\Kt_0(\Var^{\le n-1}/\k)$, not the full Grothendieck
ring as the target of $\ti{c}|_{\Bir(X)}$.
Indeed, consider the homomorphism
\begin{equation}\label{eq:iota}
\Kt_0(\Var^{\le n-1}/\k) \overset{\iota_{n-1}}\to \Kt_0(\Var^{\le n}/\k).
\end{equation}
The following lemma shows in particular
that $\iota_{n-1} \circ \ti{c}|_{\Bir(X)} = 0$.

\begin{lemma}\label{lem:kernelXY}
The kernel of \eqref{eq:plusprojectionU}
admits the following presentation:
\begin{equation}\label{eq:kernel-XY}
\langle \ \left( \ti{c}(\phi), 
-[Y] + [X] \right) \ \rangle_{\phi \in \Bir(X,Y)}
\end{equation}
for all birational isomorphisms $\phi$
between all irreducible $n$-dimensional
varieties $X$, $Y$.
\end{lemma}
\begin{proof}
The subgroup \eqref{eq:kernel-XU} can be rewritten as
\begin{equation}\label{eq:kernel-UX}
\langle \ \left( \ti{c}(j:U \hto X), 
- [X] + [U] \right) \ \rangle_{U \subset X}
\end{equation}
where 
$U \subset X$ runs over 
open subsets of all irreducible $n$-dimensional
varieties $X$. 
It remains to show that the subgroups \eqref{eq:kernel-UX}
and \eqref{eq:kernel-XY} coincide.
It is clear that \eqref{eq:kernel-UX} is a subgroup of
\eqref{eq:kernel-XY}. Conversely, every element of
\eqref{eq:kernel-XY} can be written as a combination
of elements
from \eqref{eq:kernel-UX} after decomposing
$\phi$ as a composition of open embeddings and their inverses.
\end{proof}

\begin{proposition}\label{prop:ker-iota}
We have
\begin{equation}\label{eq:ker-iota}
\Ker(\iota_{n-1}) = \sum_{X \in \Bir_n/\k} \ti{c}(\Bir(X))
\end{equation}
so that there is an exact sequence
\begin{equation}\label{eq:seq-K0}
0 \to \Image(\ti{c})_n \to \Kt_0(\Var^{\le n-1}/\k) \overset{\iota_{n-1}}\to \Kt_0(\Var^{\le n}/\k) \overset{\pi_n}\to \Z[\Bir_n/\k] \to 0,
\end{equation}
where we write $\Image(\ti{c})_n$ for the right hand side
of \eqref{eq:ker-iota}.
\end{proposition}
\begin{proof}
$\Ker(\iota_{n-1})$ coincides with kernel of $\eqref{eq:plusprojectionU}$
intersected with $\Kt_0(\Var^{n-1}/\k)$.
Hence by Lemma \ref{lem:kernelXY} every element in
$\Ker(\iota_{n-1})$ has the form
\begin{equation}\label{eq:sum-tcphi}
\sum_{i=1}^r \tc(\phi_i)
\end{equation}
with $\phi_i: X_i \dto Y_i$
and $\sum_{i=1}^r ([X_i] - [Y_i]) = 0 \in \Z[\Var^n/\k]$.
Since the latter is a free abelian group on the isomorphism
classes of $n$-dimensional irreducible varieties,
there is a permutation $\sigma \in S_r$ such that
$Y_i \simeq X_{\sigma(i)}$. 
Let $i \in \{1, \dots, r\}$ and let $\ell$ be the length
of $\sigma$-orbit of $i$. Then we have a composition
\[
\psi_i: X_i \overset{\phi_i}{\dto} X_{\sigma(i)} \overset{\phi_{\sigma(i)}}{\dto} X_{\sigma^2(i)} \overset{\phi_{\sigma^2(i)}}{\dto} \dots
\overset{\phi_{\sigma^{\ell-1}(i)}}{\dto}
X_{\sigma^\ell(i)}  =  X_i. 
\]
This allows to rewrite \eqref{eq:sum-tcphi}
as a sum of $\tc(\psi_i)$, with $\psi_i \in \Bir(X_i)$,
and $i$ running over a set of representative classes
for orbits of $\sigma$ on $\{1, \dots, r\}$.
\end{proof}

\begin{remark}\label{rem:iota-injectivity}
Using Proposition \ref{prop:ker-iota}
and
Theorem \ref{thm:surfaces-main},
one can prove that 
$\iota_0$ and
$\iota_1$ are injective
for all perfect fields $\k$ \cite[Corollary 3.10]{LSZ20}.

On the other hand when $n \ge 2$, geometric constructions in the following
section 
shows in an explicit way that $\iota_n$ is not injective
over various fields $\k$;
see Theorems \ref{thm:ell-curves}, \ref{thm:K3}, 
\ref{thm:CY3}, and Theorem~\ref{thm:bir-generation}.
\end{remark}

\section{Geometric constructions}
\label{sec:geometric-constructions}

\subsection{Strong birational isomorphisms}
\hfill

Recall that two varieties $X$ and $Y$ are isomorphic
in codimension one if there exists a birational
map $\gamma: X \dto Y$ such that
$\Ex(\gamma)$ and $\Ex(\gamma^{-1})$ have codimension $\ge 2$.
We introduce a motivic version of this notion:

\begin{definition}\label{def:strong-birat}
We say that two varieties
$X$ and $Y$ are {\it strongly
birational}
if there exists $\gamma \in \Bir(X,Y)$
such that 
$c(\gamma) = 0$.
We say that $X$ is 
{\it strongly rational}
if $X$ is strongly birational to $\P^n$.
\end{definition}

Note that strong birationality
is an equivalence relation by the
additivity of $c$ (Lemma \ref{lem-homc}).
Furthermore, isomorphism in codimension one
implies strong birationality.

\medskip

When $\k$ is a field of characteristic zero,
it follows from Proposition
\ref{prop:Picnb} that strongly birational varieties
have equal ranks of their
geometric N\'eron--Severi groups.
In particular, a strongly rational
variety has geometric N\'eron--Severi group 
 of
rank one. 

The following lemma provides some examples 
of strongly rational varieties.

\begin{lemma}\label{lem:GP-strongly-rational}
If $X$ is a smooth projective variety
containing an irreducible rational
divisor $D \subset X$
such that $U := X \setminus D$
is isomorphic to $\A^n$,
then $X$ is strongly rational.
In particular, if $G$ is a split reductive
algebraic group over a field $\k$,
then every 
homogeneous space
$X = G/P$
of Picard rank $\rho(X) = 1$
is strongly rational.
\end{lemma}
\begin{proof}
Consider the birational map 
$\phi :  X \dto \P^n$
induced by the isomorphism $U \simeq \A^n$.
We have
\[
\phi: X \hlto U \simeq \A^n \hto \P^n,
\]
so $c(\phi) = [\P^{n-1}] - [D] = 0$.
Thus $X$ is strongly rational.

Now we prove the second statement. 
A homogeneous space $X$ admits
the Bruhat stratification
into affine spaces \cite[Proposition 1.3]{Kock}. 
It has a dense open stratum isomorphic to $\A^n$ and 
the number of codimension one strata 
$\A^{n-1}$ in the Bruhat stratification
is equal to the Picard rank $\rho(X) = 1$ by~\cite[Example 1.9.1]{Fulton}.
The closure of the
$\A^{n-1}$ stratum  
is therefore the complement of the dense stratum $\A^n$.
It follows from the first part of the lemma that $X$
is strongly rational.
\end{proof}

\subsection{L-links}
\hfill

Let $X$, $Y$ be 
irreducible geometrically reduced
varieties.

\begin{Def} 
An \emph{L-link} with centers $X$ and $Y$
is a diagram 
\begin{equation}\label{eq:L-link}
 \xymatrix{
    & & T \ar[dl]_{\Bl_{X}} \ar[dr]^{\Bl_{Y}} & & \\
    X \ar@{^{(}->}[r]^{\text{closed}} & \XX \ar@{-->}[rr]^\psi &  & \XX' & \ar@{_{(}->}[l]_{\text{closed}} Y\\
}
\end{equation}
of blow ups with centers $X$ and $Y$,
subject to the following conditions:
\begin{enumerate}
    \item[(L1)] $\XX$ and $\XX'$ are 
 varieties smooth in codimension one,
 such that $[\XX] = [\XX']$ in $\Kt_0(\Var/\k)$.
    \item[(L2)] 
    $\XX$ and $\XX'$ are strongly birational.
    \item[(L3)]
    The exceptional divisors of
    $\Bl_{X}$ and $\Bl_{Y}$ are irreducible.
\end{enumerate}
\end{Def}

We say that the L-link $\psi$~\eqref{eq:L-link} is smooth 
(resp. projective)
if $X$, $Y$, $\cX$, $\cX'$, $T$
are smooth (resp. projective).
An L-link $\psi$
is called \emph{motivically non-trivial} if $c(\psi) \ne 0$.
For instance, this is the case when 
the link is smooth and
$X$ is not stably birational to $Y$.

\begin{example}
Special Cremona transformations~\cite{Crauder-Katz} give rise
to examples of L-links with
$\cX = \cX' = \P^n$
and $\gamma = \Id$. 
\end{example}

L-links are closely related to L-equivalence.
Recall that smooth
projective connected
varieties
$X$, $Y$
are
called
L-equivalent
if $\L^d([X]-[Y]) = 0$
in $\Kt_0(\Var/\k)$
for some $d \ge 0$.

\begin{lemma}\label{lemma:L-link}
Given an L-link \eqref{eq:L-link},
let $m = \codim_X(\XX)$,
$m' = \codim_Y(\XX')$. 
\begin{enumerate}
\item If the link is smooth, we have 
$\L([\P^{m-2}][X] - [\P^{m'-2}][Y]) = 0$ in $\Kt_0(\Var/\k)$.
\item Let $\phi = \gamma^{-1} \psi 
\in \Bir(\XX)$, where $\gamma$ comes from Definition \ref{def:strong-birat}.
Then 
$$c(\phi) = c(\psi) = [\P^{m-1} \times X] - [\P^{m'-1} \times Y]$$
in $\Z[\Bir/\k]$.
In particular, if the link \eqref{eq:L-link} is motivically non-trivial
(e.g. if $X$ and $Y$ are not stably
birational), then $c(\phi) \ne 0$.
\end{enumerate}
\end{lemma}
\begin{proof}
(1) follows by computing the class
of $[T]$ in two ways using the two
blow up morphisms and (L1).
For (2), we have $c(\phi) = c(\psi)$ by
(L2). 
Since $\cX$ is smooth in codimension one
and $X$ is geometrically reduced,
there exists a closed subscheme $Z \subset \cX$ of codimension two
such that both $\cX \bss Z$ and $X \bss Z$ are smooth.
As the exceptional divisor $E$ of $\Bl_{X}$
is irreducible by (L3), 
it follows that $E$ is birational to $\P^{m-1} \times X$.
Similarly,
the exceptional divisor of $\Bl_{Y}$
is birational to $\P^{m'-1} \times Y$.
Thus by Lemma~\ref{lem-homc}, we have
$c(\phi) = c(\psi) =  [\P^{m-1} \times X] - [\P^{m'-1} \times Y]$.
\end{proof}

\subsection{Elliptic curves}\label{ssec-dim3}
\hfill

Let $C$ be a smooth projective
connected
curve and let $k \in \Z$. 
Let $\Jac^k(C)$ denote
the Jacobian of divisors of degree $k$ on $C$.
Each $\Jac^k(C)$ is a torsor under
the algebraic group
$\Jac^0(C)$. 
If $C$ is a curve of genus one,
then
$\Jac^k(C)$
are twisted forms of $C$
which are not necessarily isomorphic to it.
We concentrate on curves $C$
of genus one and degree five for which
the only interesting Jacobian is $\Jac^2(C)$ \cite{ShinderZhang}.

\begin{proposition}\label{prop:elliptic}
Let $C$ be a genus one curve 
having
a divisor of degree $5$,
and let $C' = \Jac^2(C)$.
We have the following smooth projective L-link:
\begin{equation}\label{eq:L-link-curves}
 \xymatrix{
    & & T \ar[dl]_{\Bl_{C}} \ar[dr]^{\Bl_{C'}} & & \\
    C \ar@{^{(}->}[r] & Q^3 \ar@{-->}[rr]^\psi &  & \P^3 & \ar@{_{(}->}[l] C'.\\
}
\end{equation}
Here $Q^3$ is a smooth quadric threefold.
\end{proposition}
\begin{proof}
By Riemann-Roch, the linear
system of a divisor
of degree $5$ on $C$ 
gives rise to a degree $5$ embedding $C \subset \P^4$,
and it is well-known that 
$C$ is the scheme-theoretic intersection 
of the quadrics containing it, see e.g. \cite[Proposition 4.2(ii)]{Fisher-quintics}. 
Thus the map given by the linear system $|2H - C|$
is resolved by blowing up $C$, and we obtain
the Crauder--Katz quadro-cubic transformation \cite[Theorem 2.2(ii)]{Crauder-Katz}
\begin{equation}\label{eq:L-link-curves-big}
 \xymatrix{
    & & T' \ar[dl]_{\Bl_{C}} \ar[dr]^{\Bl_{S}} & & \\
    C \ar@{^{(}->}[r] & \P^4 \ar@{-->}[rr]^{\psi'} &  & \P^4 & \ar@{_{(}->}[l] S \\
}
\end{equation}
where $S$ is the surface
parameterizing the secant lines of $C$. 
Thus $S$ is isomorphic to $\Sym^2(C)$ 
(cf Proof of \cite[Theorem 3.3(A)]{Crauder-Katz})
and
is embedded as a quintic ruled surface $\Sym^2(C) \to \Jac^2(C)$
into $\P^4$.

To obtain the L-link 
\eqref{eq:L-link-curves} we take a hyperplane
section of $\eqref{eq:L-link-curves-big}$ as follows.
Take any smooth quadric $Q^3 \subset \P^4$ containing $C$
and restrict $\psi'$ onto $Q^3$. 
The image $\psi'(Q^3)$
is a hyperplane $\P^3$, 
and we obtain a birational map $\psi: Q^3 \dto \P^3$ and a diagram
\eqref{eq:L-link-curves} with $C' = S \cap \P^3$.
This curve is a section of the projection $S \to \Jac^2(C)$,
hence $C' \simeq \Jac^2(C)$,
and $C' \subset \P^3$
is a degree five genus one curve.

We claim that $Q^3$ contains $\k$-lines. 
If $\k$ is a finite field, then this is true
for any smooth quadric in $\P^4$, as the quadratic form
in 5 variables over a finite field is isomorphic to a direct
sum of two copies of the hyperbolic plane and a one-dimensional
form, by the Chevalley--Warning theorem. 
Let us assume that $\k$ is infinite.
Under the map $\psi$, general lines on $Q^3$ correspond to general
conics in $\P^3$ which are five-secant to $C'$. Such conics
are parameterized by planes $\P^2 \subset \P^3$, and taking
a general 
plane $\P^2 \subset \P^3$
defined over $\k$ we get a $\k$-rational line on $Q^3$.

Finally \eqref{eq:L-link-curves}
satisfies the definition of L-link
because $[Q^3] = [\P^3]$ as $Q^3$ 
contains a line~\cite[Example 2.8]{KuznetsovShinder},
and
$Q^3$ 
is strongly rational
(see e.g. Lemma \ref{lem:GP-strongly-rational}).
\end{proof}

\begin{theorem}\label{thm:ell-curves}
Let $\k$ be a 
field
of one of the following types:
\begin{itemize}
    \item a number field,
    \item a function field over a field $\F$,
    where $\F$ can be
    \begin{itemize}
        \item an algebraically closed field, 
        \item a number field,
        \item a finite field. 
    \end{itemize}
\end{itemize}
Then we can choose a curve
$C$ such that \eqref{eq:L-link-curves} is motivically non-trivial.
Furthermore, 
the set $I$ of isomorphism classes of such curves has the cardinality of $\k$.
\end{theorem}

Here by a function field, we mean
a function field of a positive-dimensional
$\F$-variety.

\medskip

Before we prove the theorem, we need a preliminary result 
on the arithmetic of elliptic curves.

\begin{lemma}\cite{Lang-Tate, ClarkLacy}
\label{lem:torsors}
For a field $\k$ 
as in Theorem \ref{thm:ell-curves},
there exist infinitely many 
genus one
curves $C$
without rational
points but having 
a divisor of degree $5$,
and such that $j(\Jac^0(C)) \ne 1728$.
The set $I$ of isomorphism classes of such curves has the cardinality of $\k$.
\end{lemma}
\begin{proof}
If $\k$ is an infinite 
finitely generated field over $\Q$ or over a finite field, 
then infinitely many
such curves exist by~\cite[Theorem 1.10]{ClarkLacy};
the condition $j(\Jac^0(C)) \ne 1728$ is not explicitly
checked in \cite{ClarkLacy}, however the elliptic
curves they use to produce torsors indeed satisfy this condition \cite[Section 4]{ClarkLacy}.

The remaining case is
a function field $\k$ over an algebraically
closed field $\F$. 
We apply~\cite[Theorem 7]{Lang-Tate}.
First of all since $\k$ is Hilbertian~\cite[Proposition 12.3.3, Theorem 13.4.2]{MR2445111},
it has infinitely many cyclic extensions
of degree $5$ which are
linearly disjoint in $\ol{\k}$~\cite[Corollary 16.2.7]{MR2445111}. 
In particular, these cyclic extensions are non-isomorphic.
Next we take any ordinary elliptic
curve $E/\F$
with $j(E) \ne 1728$ in $\F$.
Since $\F$ is algebraically closed,
$E(\F)$ has non-trivial $5$-torsion.
The base change $E_\k$ of $E$ to $\k$
has also non-trivial $5$-torsion
and $j(E_\k) \ne 1728$ in $\k$.
The group $E(\k) / 5E(\k)$
is finite by~\cite[Remark 1.3, Proposition 1.4]{ClarkLacy}. 
We thus conclude by~\cite[Theorem 7]{Lang-Tate}
that there exists infinitely many $E_\k$-torsors of index $5$.

Note that among the fields we consider, $\k$ is uncountable only if $\k$ is the function field over an algebraically closed field $\F$.
In this case $\F$ and $\k$ have the same cardinality.
As the set of elliptic curves $E/\F$
considered above has the same cardinality as $\F$,
the last statement follows.
\end{proof}

\begin{proof}[Proof of Theorem \ref{thm:ell-curves}]
Take the curve $C$ defined
by Lemma \ref{lem:torsors}
and set $C' = \Jac^2(C)$.
By \cite[Lemma 2.7]{ShinderZhang}
we have $C \not\simeq C'$
(the cited result assumes $\chr(\k) = 0$,
however the proof works without this assumption
since
there is no $\gs \in \Aut(E)$ of order $4$ when $j \ne 1728$
for any field
\cite[Proposition A.1.2]{SilvermanAEC}).
As $C$ and $C'$ are non isomorphic smooth curves, they are
not birational.
Hence the L-link~\eqref{eq:L-link-curves} is motivically non-trivial.
\end{proof}

\begin{remark}
The L-link \eqref{eq:L-link-curves} appears in the classification
of Fano threefolds of rank two by Mori--Mukai \cite[2) on page 117]{Mori-Mukai}.
It is a very interesting question whether
there exist other
nontrivial smooth projective
L-links between strongly rational varieties in dimension $3$. The ultimate
question in this direction
is to fully describe the image $c(\Bir(\P^3))$
using the Sarkisov link decomposition.
\end{remark}

\begin{remark}
Let $X \to B$ be an elliptic surface
with a multisection of degree five over $\k = \C$.
Let $Y = \Jac^2(X/B)$ be the relative Jacobian
of degree two divisors.
Then spreading out the construction
of Theorem \ref{thm:ell-curves},
one can show that
$[X] - [Y] \in \Image(c_{\P^3 \times B})$.
See~\cite{ShinderZhang}
for a detailed analysis when $X$, $Y$
are not birational in the case
of elliptic K3 surfaces.
\end{remark}

\subsection{K3 surfaces of degree $12$} 
\hfill

\begin{theorem}\label{thm:K3}
Let $\k \subset \C$.
There exist motivically 
non-trivial L-links 
\begin{equation}\label{eq:L-link-K3}
 \xymatrix{
    & & T \ar[dl]_{\Bl_{S_0}} \ar[dr]^{\Bl_{S'_0}} & & \\
    S_0 \ar@{^{(}->}[r] & \P^4 \ar@{-->}[rr]^\psi &  & \P^4 & \ar@{_{(}->}[l] S'_0\\
}
\end{equation}
where $S_0$ and $S'_0$ are projective
surfaces such that the minimal models $S$ and $S'$ of their normalizations
are non-isomorphic
K3 surfaces of geometric Picard
rank one and degree $12$, 
and $S'_\C$ is the unique Fourier--Mukai partner of $S_\C$.
Furthermore, 
the set $I$ of isomorphism classes of the K3 surfaces $S$ occurring in these L-links has the cardinality of $\k$.
\end{theorem}

Theorem \ref{thm:K3}, over $\C$,
is the main result of Hassett and Lai~\cite{HassettLai}.
We start by summarizing their construction,
and then we explain how to extend it to an arbitrary
subfield $\k \subset \C$ using Lefschetz pencils.

\begin{theorem}[{\cite[Theorem 2.1, Theorem 3.1]{HassettLai}}]\label{thm-HL}
Let $(S,H)$ be a general polarized complex K3 surface with $H^2 = 12$
and let $\gS = \Set{ x_1,x_2,x_3 } \subset S$ be a general triple of points.
The linear system $|H|$ defines an embedding $S \hto \P^7$;
let $S_0 \subset \P^4$ be the proper transform
of $S$ under the projection $\P^7 \dto \P^4$
from the plane generated by $\gS$.
We have the following:
\begin{enumerate}
\item $S_0$ has three transverse double points,
and the normalization of $S_0$ is $\Bl_\gS S$.
\item The linear system of quartics containing $S_0$ cuts out
$S_0$ as a scheme, and
defines a birational map
$\psi: \P^4 \dto \P^4$.
\item 
The inverse
$\psi^{-1}: \P^4 \dto \P^4$
is constructed the same way,
starting from another K3 surface $S'$ of degree $12$
and three points $\gS' = \Set{ x_1', x_2', x_3' }$ on $S'$.
\item We have $\Bl_{S_0}\P^4 \simeq \Bl_{S'_0}\P^4$
and it resolves $\psi$,
where $S'_0$ the proper transform
of $S'$ under the projection $\P^7 \dto \P^4$
from the plane generated by $\gS'$.
\end{enumerate}
Furthermore, if $\Pic(S) = \Z \cdot H$, 
then $S'$ is the unique Fourier--Mukai partner of $S$ such that
$S \not\simeq S'$.
Finally, if $S \hto \P^7$ and $\gS \subset S$ are defined over a subfield $\k \subset \C$,
then both $\psi$ and the pair $(S',\gS')$ are defined over $\k$.
\end{theorem}
\begin{proof}
The main statement follows from~\cite[\S1.1, Theorem 2.1]{HassettLai}.
The statement about $S \not\simeq S'$ follows from~\cite[Theorem 3.1]{HassettLai},
together with the derived invariance of Picard number for K3 surfaces 
and~\cite[Proposition 1.10]{OguisoFM}.
Finally, if $S \hto \P^7$ and $\gS \subset S$ are defined over $\k$,
then $\psi$ is defined over $\k$ by construction.
As $S_0'$ is the fundamental locus of $\psi^{-1}$,
$S_0'$ is defined over $\k$ as well.
Since $S'$ is the minimal model of the normalization of $S'_0$,
and $\gS' \subset S'$ is the image of the $(-1)$-curves,
we conclude that $(S',\gS')$ is defined over $\k$.
\end{proof}

Let $\k \subset \C$.
Let $S \subset \P^7$ be a K3 surface
of degree $12$ defined over $\k$ 
and $\Sigma \subset S$
a smooth subscheme of length three.
We say that $(S, \Sigma)$ is HL-admissible
if 
Theorem \ref{thm-HL}(1-4) holds
    for $S_\C$ and $\Sigma_\C$.

It is a construction going back to Mukai
that 
general 
K3 surfaces of degree $12$ can be obtained
as linear sections 
of a Grassmannian
$\OG_+(5,10) \subset \P^{15}$, see \cite[2.1]{HassettLai}.
Consider a general $3$-dimensional
linear section $V \subset \P^8$ over $\Q$
of this Grassmannian.
As $\OG_+(5,10)$, being a rational homogeneous variety, 
has dense $\Q$-points and since the linear section cutting off $V$ is general,
we can assume that $V$ is smooth and has a $\Q$-point. 
This implies that $V$ is $\Q$-rational~\cite[Theorem 1.1]{KuPrRatFano3}. 

Let $\P^6 \subset \P^8$ be a linear subspace
and
$\Sigma = \{ x_1, x_2, x_3 \} \subset V \cap \P^6$.
We call the 
pair 
$(\P^6, \gS)$ \emph{HL-admissible}
if the following conditions are satisfied:
\begin{enumerate}
    \item[(a)] The pencil $S_t$, $t \in \P^1$, 
    defined by intersecting
    $V$ with hyperplanes containing $\P^6$ is a Lefschetz pencil.
    \item[(b)] For one (hence for general) $t$, the pair
    $(S_t, \Sigma)$ is HL-admissible.
\end{enumerate}

\begin{lemma}\label{lem:HL-admis-pencil}
For a general $V$ as above, 
there exists an HL-admissible pair $(\P^6, \Sigma)$
defined over $\Q$.
\end{lemma}
\begin{proof} 
As both conditions (a) and (b) are nonempty Zariski open conditions,
HL-admissible pairs form a Zariski dense 
open subset $U$
in the incidence
variety 
$$\cX \cnec \Set{(x_1,x_2,x_3; L) \in V^3 \times \Gr(\P^6,\P^8)
| x_1,x_2,x_3 \in V \cap  L}.$$
It is easy to see that $\cX$ is a $\Q$-rational variety,
hence the open subset $U \subset \cX$ has (dense)
$\Q$-rational points.
\end{proof}

\begin{proposition}\label{pro-K3HLPic1}
There exist infinitely many polarized K3 surfaces $(S,H)$ of degree $12$ over $\Q$
together with three $\Q$-points $\gS = \Set{x_1,x_2,x_3} \subset S(\Q)$
such that
\begin{enumerate}
    \item $\Pic(S_\C) \simeq \Z $;
\item $(S; \gS)$ is HL-admissible. 
\end{enumerate}
The same conclusion holds if $\Q$ is replaced by any subfield $\k$ of $\C$,
and the isomorphism classes of such K3 surfaces $S$
form a set $I$ which has the cardinality of $\k$.
\end{proposition}

\begin{proof} 
The Lefschetz pencil $\{S_t\}$ 
given by the HL-admissible pair from 
Lemma~\ref{lem:HL-admis-pencil}
is defined over $\Q$,
hence Terasoma's argument in~\cite[Step 2 and Step 3 in \S3]{Terasoma}
proving~\cite[Theorem 1]{Terasoma}
applies verbatim 
and shows that there exist infinitely many 
K3 surfaces among $\{S_t\}$ defined over $\Q$
(and thus over any subfield $\k \subset \C$) with 
geometric Picard rank $1$. 
If $\k$ is uncountable,
then since ``having geometric Picard rank $1$" is a 
very general property,
the set of K3 surfaces over $\k$ 
among $\{S_t\}$ for which (1) holds
is also uncountable.

We claim that
the Lefschetz pencil
$\{S_t\}$ is non-isotrivial.
This is because the monodromy action on $H^2(S_{t,\C}, \Q)$
is nontrivial by the invariant cycle theorem,
and $\Aut(S_{t,\C}) = \{\Id\}$ if $\rho(S_{t,\C}) = 1$~\cite[Corollary 15.2.12]{HuybK3book}, 
hence
monodromy does not act by holomorphic transformations.
Thus $\{S_t\}$ contains infinitely many
isomorphism classes of K3 surfaces for any $\k \subset \C$,
uncountably many if $\k$ is uncountable.
Finally, since 
the subset of K3 surfaces $S_t$ 
in $\{S_t\}$
for which $(S_t; \gS)$ is HL-admissible  
is open, 
Proposition~\ref{pro-K3HLPic1} follows.
\end{proof}

\begin{proof}[Proof of Theorem \ref{thm:K3}]
Let $S$ be a K3 surface over $\Q$ as in Proposition~\ref{pro-K3HLPic1}
and let $S'$ be as in Theorem \ref{thm-HL}.
In particular, $S_\C \not\simeq S'_\C$ 
and $S'_\C$
is a Fourier-Mukai
partner of $S_\C$.

Let $\psi \in \Bir(\P^4)$ be the map as in Theorem~\ref{thm-HL}. 
By Theorem~\ref{thm-HL},
both $\psi$ and the K3 surface $S'$ are defined over $\Q$.
Since $S_0$ has at worst transversal double points as singularities,
the exceptional divisor of $\Bl_{S_0} : T \to \P^4$
is birational to $S_0 \times \P^1$,
which is further birational to $S \times \P^1$
by Theorem~\ref{thm-HL}(1).
Similarly, 
the exceptional divisor of $\Bl_{S'_0} : T \to \P^4$
is birational to $S' \times \P^1$.
Thus,
$$c(\psi_\k) = [S_\k \times \P^1_\k] - [S'_\k \times \P^1_\k] \in \Z[\Bir/\k].$$
Since $S_\k$ and $S_\k'$ are non-isomorphic
K3 surfaces 
because $S_\C \not\simeq S'_\C$, 
they are not stably birational, 
see e.g. Corollary~\ref{cor-CYnSB}.
Hence $c(\psi_\k) \ne 0$,
so $\psi_\k$ is motivically non-trivial.
\end{proof}

\subsection{K-trivial
threefolds}\label{ssec-IMOU}
\hfill

\begin{theorem}\label{thm:CY3}
Let $\k$ be an infinite field.
There exist motivically non-trivial smooth projective L-links:
\begin{equation}\label{eq:L-link-CY3}
 \xymatrix{
    & & T \ar[dl]_{\Bl_{Z}} \ar[dr]^{\Bl_{Z'}} & & \\
    Z \ar@{^{(}->}[r] & \XX \ar@{-->}[rr]^\psi &  & \XX' & \ar@{_{(}->}[l] Z'\\
}
\end{equation}
where $\XX$ and $\XX'$ are five-dimensional
$G_2$-Grassmannians, and
$Z$ and $Z'$ are 
K-trivial
threefolds
of Picard rank $1$.
\end{theorem}

Here, a K-trivial variety is a smooth projective variety $X$
with $\go_X \simeq \cO_X$.

\medskip

The construction of Theorem \ref{thm:CY3}
is given by Ito--Miura--Okawa--Ueda~\cite{IMOUG2} 
over a field of characteristic zero. 
We explain that the construction extends
to an arbitrary infinite field.
The proof of Theorem \ref{thm:CY3} occupies
the rest of the section.

Let $G$ be the connected and simply connected 
split simple group scheme of type $G_2$ over $\Z$.
Let $B \subset G$
be the Borel subgroup, and
$P,P' \subsetneq G$ 
the two maximal parabolic subgroups of $G$ containing $B$.
The quotients $G/B$, $G/P$, and $G/P'$ 
are smooth and projective over $\Z$ 
by~\cite[Theorem 13.33]{MilneAlgGp} 
and~\cite[Corollary XXII.5.8.5]{SGA3}.
The natural projection $p : G/B \to G/P$
is also smooth. 
Since $p_\Q : G_\Q/B_\Q \to G_\Q/P_\Q$
is a $\P^1$-fibration, so is $p : G/B \to G/P$.
Likewise, the other projection
$p' : G/B \to G/P'$ is also a smooth $\P^1$-fibration.

Our assumptions on $G$ imply that
$X^*(G) = 0$ and 
$\Pic(G) = 0$~\cite[Remark VII.1.7.a)]{RaynaudAmple},
where $X^*(G)$ is the character group of $G$.
It follows that
$$\Pic(G/P) \simeq \X^*(P) \simeq \Z$$
by~\cite[Proposition VII.1.5]{RaynaudAmple}.
Similarly, $\Pic(G/P') \simeq \Z$.
Let $M$ and $M'$ 
denote the ample generators of $\Pic(G/P)$ and $\Pic(G/P')$
respectively.
The tensor product 
$$L \cnec p^*M \otimes p'^*M'$$
is then ample as well.
Since $\deg({L_\k}_{|F_\k}) = 1$ when $\chr(\k) = 0$
where $F_\k \simeq \P^1_\k$ is a fiber of $p_\k : G_\k/B_\k \to G_\k/P_\k$
(see e.g.~\cite{IMOUG2}),
the same holds true for any field $\k$.
It follows that 
$E \cnec p_*L$ is locally free of rank $2$ and
$G/B \simeq \P(E)$ over $G/P$.

Now let $\k$ be an infinite field
and let $\gs \in H^0(G_\k/B_\k,L_\k)$ be a general section. 
Let $X_\k = Z(\gs) \subset G_\k/B_\k$ and let
$Z_\k \cnec Z(p_*\gs) \subset G_\k/P_\k$ 
with $p_*\gs$
regarded as a section of $E_\k = (p_*L)_\k$.
Similarly, let $Z'_\k \cnec Z(p'_*\gs) \subset G_\k/P'_\k$.
Since $L_\k$ is very ample~\cite[Theorem 3]{RamRam} 
and $\k$ is infinite,
$X_\k$ is smooth  
by Bertini's theorem~\cite[Th\'eor\`eme 6.3]{JouanolouBertini},~\cite[Theorem 1.10]{MukaiK31820}.
Consider the commutative diagram
\begin{equation}\label{eq:G-P-diagram}
\xymatrix{
    & & X_\k \ar[ddl]_{\tau} \ar[ddr]^{\tau'}
    \ar@{->}[d] & & \\
    &  & G_\k/B_\k \ar[dl]\ar[dr] &  & \\
    & G_\k/P_\k \ar@{-->}[rr]^{\psi} 
    & & G_\k/P'_\k & 
}
\end{equation}
which defines the maps $\tau$ and $\tau'$.
By construction, 
$\tau$ maps $X_\k \bss \tau^{-1}(Z_\k)$
isomorphically onto its image
and $\tau^{-1}(Z_\k) \to Z_\k$ is a $\P^1$-fibration.
By~\cite[Theorem 2]{DanilovDecomp},
$\tau$ is thus a sequence of blow ups along smooth centers
of codimension $2$ lying above $Z_\k$.
Since $X_{\k}$ has Picard rank
$\rho(X_\k) = \rho(G_\k/B_\k) = 2$
by Grothendieck--Lefschetz' theorem,
necessarily $\tau$ is the blow up of $G_\k/P_\k$ along $Z_\k$
and $Z_\k$ is connected.
Similarly, $\tau'$ is the blow up of $G_\k/P'_\k$ along $Z'_\k$,
and $Z'_\k$ is connected.
We thus have a birational map 
$$\psi \cnec \tau' \circ \tau^{-1} : G_\k/P_\k \dto G_\k/P'_\k.$$

By construction, $\psi$ is regular away from $Z_\k$
and $\psi^{-1}$ is regular away from $Z'_\k$.

\begin{lem}\label{lem-nSBG2}
$Z_\k$, $Z'_\k$ are 
K-trivial threefolds,
which are not stably birational to each other.
\end{lem}

\begin{proof}
When $\chr(\k) = 0$, Lemma~\ref{lem-nSBG2} is proven in~\cite{IMOUG2}.
Assume that $\chr(\k) = p > 0$.
Let $R$ be an integral local ring of characteristic zero with residue field $\k$
(see~\cite[Satz 20]{HasseSchmidt} or~\cite[Theorem 2.10]{HSbis} for the existence).
Since $H^j(G_\k/B_\k,L_\k)=0$ for all $j > 0$~\cite[Theorem 2]{RamRam} (and for any field $\k$),
we have
$$H^0(G_\k/B_\k,L_\k) \simeq H^0(G_R/B_R,L_R) \otimes_R \k$$
by Grauert's base change.
Thus the section $\gs \in H^0(G_\k/B_\k,L_\k)$
defining $Z_\k$
can be lifted to a section $\gs_K \in H^0(G_K/B_K,L_K)$
where $K \cnec \Frac (R)$, which is of characteristic zero.
This defines a lifting of $Z_\k$ to a smooth projective subvariety 
$Z_K \subset G_K/P_K$.  
Since $\go_{Z_K}$ is trivial, so is $\go_{Z_\k}$.
The same argument shows that $\go_{Z'_\k}$
is trivial.

To show that $Z_\k$, $Z'_\k$ are not stably birational,
we may assume $\k = \ol{\k}$. First
we show that $Z_{\k}$ and $Z'_{\k}$
are not isomorphic.
We have seen that $\rk\,\NS(X_{\k}) = 2$.
As $X_{\k} \to Z_{\k}$ and $X_{\k} \to Z'_{\k}$ are $\bP^1$-bundles,
necessarily $\rk\,\NS(Z_{\k})= \rk\,\NS(Z'_{\k}) = 1$.
Let $H$ (resp. $H'$) 
be the ample generator of $\NS(Z_{\k})$ (resp. $\NS(Z'_{\k})$).
Let $a, a' \in \Z_{>0}$ such that
$$c_1({M_{\k}}_{|Z_{\k}}) = aH \ \ \ \text{ and } \ \ \
c_1({M'_{\k}}_{|Z'_{\k}}) = a'H'.$$
As 
$$c_1({M_{\k}}_{|Z_{\k}})^3 = c_1(M_{\k})^3 \cdot Z_{\k} 
= c_1(M_K)^3 \cdot Z_K
\ \ \
\text{ and }  \ \ \ c_1({M'_{\k}}_{|Z'_{\k}})^3 = c_1(M'_{\k})^3 \cdot Z'_{\k} = c_1(M'_K)^3 \cdot Z'_K$$
are distinct and equal to 
$14$ or $42$ by~\cite{IMOUG2},
necessarily $a = a' =1$.
Hence $H^3 \ne H'^3$,
so $Z_{\k}$ is not isomorphic to $Z'_{\k}$.
It follows from Corollary~\ref{cor-CYnSB}
that $Z_\k$ is not stably birational to $Z'_\k$.
\end{proof}

\begin{proof}[Proof of Theorem \ref{thm:CY3}]
We set $\XX = G_\k/P_\k$
and $\XX' = G_\k/P'_\k$.
Both $\XX$ and $\XX'$ have a 
Bruhat decomposition
with a single cell in every dimension so that
$[\XX] = [\P^5] = [\XX']$.
Furthermore, $\XX$ and $\XX'$ are strongly
birational, because they are both
strongly rational by Lemma~\ref{lem:GP-strongly-rational}.
Thus diagram \eqref{eq:G-P-diagram} gives us the L-link
\eqref{eq:L-link-CY3}.
Finally, since  
$Z_\k$ and $Z'_\k$ are not (stably) birational by Lemma~\ref{lem-nSBG2},
the L-link~\eqref{eq:L-link-CY3} is motivically non-trivial.
\end{proof}

\begin{remark}
In the proof of Theorem~\ref{thm:CY3}, 
the only place where we need the assumption that
$\k$ is infinite is to guarantee 
the existence of $\gs \in H^0(G_\k/B_\k,L_\k)$
such that the vanishing loci $Z_\k \subset G_\k/P_\k$ and $Z'_\k \subset G_\k/P'_\k$
of $p_*\gs$ and $p'_*\gs$ 
are smooth. 
Based on generic smoothness,
the same conclusion of Theorem~\ref{thm:CY3} thus holds
for all but finitely many finite fields $\k$.
\end{remark}

\section{Applications to Cremona groups}
\label{sec:cremona}

Let $X/\k$ be a variety.

\begin{definition}\label{Def-reg}
A birational automorphism $\phi \in \Bir(X)$ is pseudo-regularizable
if $\phi = \alpha^{-1} \circ \gamma \circ \alpha$ where
\begin{itemize}
    \item $\alpha: X \dashrightarrow X'$ is a birational map
    (we do not assume $X'$ is smooth nor projective);
    \item $\gamma \in \Bir(X')$ is a pseudo-automorphism, namely
    an isomorphism in codimension one.
\end{itemize}

If moreover the map $\gamma$
is biregular, we call $\phi$ a regularizable map.
\end{definition}

Pseudo-regularizable maps have been
studied in \cite{Cantat-Cornulier, Lonjou-Urech}
in the context of the actions of Cremona
groups on various combinatorial complexes.

\begin{example} \label{ex-regbir} \hfill
\begin{enumerate}
    \item Regular automorphisms, or more generally
birational maps
$\phi \in \Bir(X)$ which have an invariant dense open subset
$U \subset X$, are regularizable.
\item Finite order elements of $\Bir(X)$ are regularizable.
Indeed, given such 
an element $\phi \in \Bir(X)$, the complement 
\[
U := X \setminus \bigcup_{k = 1}^{\ord(\phi)-1} \Ex(\phi^k),
\]
is an invariant dense open subset, and 
we apply (1) to conclude.
By Weil's regularization theorem (see also~\cite{Kraft}), 
they are even projectively regularizable,
namely there exists a regularization $\alpha : X \dto X'$
with projective $X'$.
\end{enumerate}
\end{example}

Let us write 
$\Bir(X)^{\psreg}$ for the normal subgroup
generated by pseudo-regularizable elements.

\begin{lemma}\label{lem:regularizable}
We have 
\[
\Bir(X)^{\psreg} \subset \Ker(c|_{\Bir(X)}).
\]
\end{lemma}
\begin{proof}
It suffices to show that
for any pseudo-regularizable element
$\phi \in \Bir(X)^{\psreg}$ we have
$c(\phi) = 0$.
By definition, for any such element 
we have $\phi = \alpha^{-1} \circ \gamma \circ \alpha$ where $\gamma$ is a pseudo-automorphism.
Thus $c(\gamma) = 0$ and 
\begin{equation*}
c(\phi) = -{c}(\alpha) + {c}(\gamma) + {c}(\alpha) = 0.
\end{equation*}
\end{proof}

\begin{theorem}\label{thm:bir-generation}
Assume that $X/\k$ is a variety
that
satisfies one of the following conditions:
\begin{enumerate}
    \item[(a)] $\k$ is a number field,
    or a function field over a number field,
    over a finite field, or over an algebraically closed field,
    and $X$ is birational to $\P^3 \times W$ for a
    separably rationally
    connected variety $W$
    (e.g. $X = \P^n$, $n \ge 3$).
    \item[(b)] $\k \subset \C$ is any subfield
    and $X$ is birational to $\P^4 \times W$ for a  rationally
    connected variety $W$
    (e.g. $X = \P^n$, $n \ge 4$).
    \item[(c)] $\k$ is any infinite
    field
    and $X$ is birational to $\P^5 \times W$ for a separably rationally
    connected variety $W$ 
    (e.g. $X = \P^n$, $n \ge 5$).
\end{enumerate}
We have $c|_{\Bir(X)} \ne 0$. 
As a consequence, 
$\Bir(X)$ is not generated by pseudo-regularizable elements. 
\end{theorem}

In particular, $\Bir(X)$ is not generated by 
any collection of elements from Example~\ref{ex-regbir}.

\begin{proof}
By Lemma~\ref{lem:regularizable}, 
it suffices
to explain that $c|_{\Bir(X)}$ is nonzero
for $\k$ and $X$ as in
the statement of the theorem.

Let us first assume that $W = \Spec(\k)$.
In case (a) (resp. (b) and (c)) we 
use the motivically non-trivial L-link constructed in Proposition~\ref{prop:elliptic} 
(resp. Theorem~\ref{thm:K3} and Theorem~\ref{thm:CY3}) ,
which together with 
Lemma~\ref{lemma:L-link} gives a map $\phi \in \Bir(\P^3)$
(resp. $\Bir(\P^4)$, $\Bir(\P^5)$)
such that 
$$c(\phi) = [Z \times \P^1] - [Z' \times \P^1] \ne 0 \in \Z[\Bir/\k].$$
where $Z$ and $Z'$ are non-isomorphic
K-trivial
varieties of Picard number $1$.
For an arbitrary $W$ as in Theorem~\ref{thm:bir-generation},
it then follows from Corollary~\ref{cor-CYnSB}
that
$$c(\phi \times \Id_W) = [Z \times \P^1 \times W] - [Z' \times \P^1 \times W] 
\ne 0 \in \Z[\Bir/\k].$$
This proves the non-vanishing of $c|_{\Bir(X)}$, and Theorem~\ref{thm:bir-generation} follows.
\end{proof}

\begin{corollary}\label{cor:homZ}
Let $X/\k$ be as in cases (a) or (b) in Theorem \ref{thm:bir-generation}.
Then we have a surjective homomorphism
$
\Bir(X) \to A
$
where $A = \bigoplus_J \Z$
is a free abelian group
with $J$ a set 
of the same cardinality as $\k$.
In particular, $A$ is a direct summand
of $\Bir(X)^{\ab}$.
\end{corollary}

\begin{proof}
We use the same
L-links which appear in the proof of
Theorem~\ref{thm:bir-generation}.
By Theorem~\ref{thm:ell-curves}
in case (a), and Theorem~\ref{thm:K3}
in case (b),
these L-links are
parameterized by a set $I$ having the cardinality of $\k$.
 
Let $X_i$ and $Y_i$ be the centers of these links.
Since the $X_i$'s are mutually non-isomorphic $K$-trivial varieties of Picard rank one, their birational types are all distinct (see e.g.
Theorem~\ref{thm:BurtCY}).
Furthermore, $X_i$ and $Y_i$
are in a symmetric relation with each other.
Namely, in case (a), the relation
is $Y_i \simeq \Jac^2(X_i)$,
so that $X_i \simeq \Jac^2(Y_i)$,
because we work with index $5$ genus one curves, see \cite[2.2]{ShinderZhang}.
In case (b), the relation is:
$Y_{i, \C}$ is the 
unique nontrivial
Fourier--Mukai partner of $X_{i,\C}$ \cite[3.1]{HassettLai}.
Thus in either case,
$$J \cnec \Set{ \Set{X_i,Y_i} \subset I | i \in I}$$
forms a partition of $I$
and has the same cardinality as $\k$.

We let
$$A \cnec \bigoplus_{\Set{X_i,Y_i} \in J}\Z\left([X_i \times \P^1 \times W] - [Y_i \times \P^1 \times W]\right) \subset
\Z[\Bir/\k].$$
By construction this a direct summand 
of $\Z[\Bir/\k]$, and let $\pi: \Z[\Bir/\k] \to A$ be the  projection homomorphism.
For each $i$ there exists
$\phi \in \Bir(X)$ such that $$\pi(c(\phi)) = 
[X_i \times \P^1 \times W] - [Y_i \times \P^1 \times W].$$
Thus the composition
$\pi \circ c: \Bir(X) \to A$ is surjective.

The statement about the abelianization follows
because $\Bir(X)^{\ab}$ admits a surjective homomorphism onto the free $\Z$-module $A$.
This homomorphism admits a splitting, hence it is 
a projection onto a direct summand.
\end{proof}

\begin{remark}
(i) After the results of this paper were announced, Blanc, Schneider and Yasinsky have proved that for all $n \ge 4$, $\Cr_n(\C)$ 
admits a surjective homomorphism onto an uncountable free product of $\Z$
\cite[Theorem B]{BSY} (which is
a stronger statement
than what
the Corollary \ref{cor:homZ} says
for $X = \P^n$ and $\k = \C$), with a completely different method.

(ii) For $n = 3$ we do not know
if there exist nontrivial homomorphisms $\Cr_3(\C) \to \Z$
(or $\Cr_3(\ol{\Q}) \to \Z$).
Whether $\Cr_3(\C)$ is generated by involutions, or regularizable elements, is an 
outstanding open
question about Cremona groups.
\end{remark}

\appendix

\section{Stably birational geometry of K-trivial varieties}
\label{app:MRC}
\medskip

The following result is well-known in characteristic zero.
We present the proof 
which works over arbitrary fields as well as a more general statement (see Remark \ref{rem:burt}), both 
communicated to us by Burt Totaro.

\begin{theorem}[Totaro]\label{thm:BurtCY}
\label{pic1}
Let $X_1$ and $X_2$ be 
smooth projective varieties
over a field with a
birational map
$\phi:X_1 \dto X_2$.
Assume that $K_{X_1}$ and $K_{X_2}$ are nef.
Then $\phi$
is a pseudo-isomorphism.
If in addition 
 $\NS(X_1)_\Q$ 
 or $\NS(X_2)_\Q$
 is one-dimensional, then
$\phi$ is an isomorphism.
\end{theorem}

\begin{proof}
Let us show that $\phi$ is a pseudo-isomorphism, that is both
$\phi$
and $\phi^{-1}$ do not have any exceptional divisors.
We follow the proof of \cite[Corollary 3.54]{KollarMori}, which works
in any characteristic. Let $Y$ be the normalization of the closure
of the graph of $\phi$,
with birational morphisms $g_1\colon Y\to X_1$
and $g_2\colon Y\to X_2$. 
We have a linear equivalence of $\Q$-Weil divisors:
$$K_Y \sim_{\Q} g_i^*(K_{X_i})+Z_i$$
for $i=1,2$, where $Z_i$ is an effective $\Q$-divisor
whose support is the union
of all exceptional divisors of $g_i$ (because $X_i$ are smooth, hence terminal \cite[Claim 2.10.4]{KollarSingMMP}). 
So we have
$$g_1^*(K_{X_1})-g_2^*(K_{X_2})\sim_{\Q} Z_2-Z_1.$$
Here $g_1^*(K_{X_1})$ and $g_2^*(K_{X_2})$ are nef.
Applying the negativity lemma to $g_1 : Y \to X_1$
(resp. $g_2 : Y\to X_2$), we obtain that $Z_2-Z_1$
(resp. $Z_1-Z_2$) is effective \cite[Lemma 3.39]{KollarMori}.
(The negativity lemma holds in any characteristic,
because the proof uses resolution
of singularities only in dimension 2.)
Thus $Z_1=Z_2$ and so $g_1$
and $g_2$ have the same exceptional divisors, hence $\phi$ is a pseudo-isomorphism.

Thus $\phi^*$
gives an isomorphism
\[
\Pic(X_2) \simeq \Pic(X_1),
\]
which takes divisors
to effective divisors (but in general does not preserve ampleness).
Furthermore, 
we have
an induced
isomorphism $\NS(X_2) \simeq \NS(X_1)$ \cite[Example 19.1.6]{Fulton}.
Now assume that
$\NS(X_1)_\Q \simeq \NS(X_2)_\Q \simeq \Q$.
Since ampleness is a numerical condition, in this case
every nonzero
effective divisor is ample, in particular
$\phi^*$ 
takes ample divisors to 
ample divisors.
Take any ample
divisor 
$H_2 \in \Pic(X_2)$ 
and let $H_1 = \phi^*(H_2)$.
For 
every $m \ge 0$,
$\phi$ induces 
an isomorphism $H^0(X_1,\OO(mH_1)) \cong H^0(X_2,\OO(mH_2))$.
These are compatible with products,
that is we have a ring isomorphism
\[
\bigoplus_{m\geq 0} H^0(X_1,\OO(mH_1)) \simeq
\bigoplus_{m\geq 0} H^0(X_2,\OO(mH_2)).
\]
By taking Proj of both sides, it follows that $\phi$ is in fact
an isomorphism from $X_1$ to $X_2$.
\end{proof}

\begin{remark}\label{rem:burt}
The proof of Theorem~\ref{thm:BurtCY} 
goes through in the relative case: given proper morphisms of varieties over a field $f_1: X_1 \to S$ and $f_2: X_2 \to S$ 
with $X_1$ and $X_2$ smooth (or just $\Q$-Gorenstein
terminal),
assume 
that
$K_{X_i}$ is nef over $S$ for $i = 1, 2$.
Let $\phi: X_1 \dto X_2$
be a birational map over $S$.
Then $\phi$ is a pseudo-isomorphism.
\end{remark}

We need to recall some results about separable maps in positive characteristic.
A dominant map $\phi: X \dto Y$ between $\k$-varieties 
is called separable if the
corresponding field extension $\k(Y)/\k(X)$
is separable~\cite[\href{https://stacks.math.columbia.edu/tag/030I}{Tag 030I}]{stacks-project},
that is $\k(Y)$ is a finite
separable extension of a purely transcendental extension of
$\k(X)$. 
Equivalently,
$\k(X)$ is geometrically
reduced over $\k(Y)$~\cite[\href{https://stacks.math.columbia.edu/tag/05DT}{Tag 05DT}]{stacks-project}.
Recall that a variety $X$ is called separably uniruled
if there exists a 
separable
dominant map
$$Y \times \P^1 \dto X$$
for some variety $Y$ with $\dim Y = \dim X - 1$
(see e.g.~\cite[Definition IV.1.1]{Kollarrat}).
The following lemma should be well-known.

\begin{lem}\label{lem-defSU}
Let $X$ be a smooth complete variety over an algebraically closed field $\k$.
Suppose that there exist a variety $Y$ and
a separable dominant map
$$u: Y \times \P^1 \dto X$$
such that the proper transform of $\{y\} \times \P^1$ 
is not a point for general $y \in Y$.
Then $X$ is separably uniruled.
\end{lem}

\begin{proof}
By~\cite[Theorem IV.1.9]{Kollarrat},
it suffices to prove the existence of a
morphism
$f : \P^1 \to X$
which is free (this means that $f^*T_X$
is globally generated).
This is shown in~\cite[Lemma 1.2]{MR977778};
for the convenience of the reader, we reproduce the proof.
Since $\k$ is assumed
to be algebraically closed
(hence perfect), after shrinking $Y$, we can assume that $Y$ is smooth.
Therefore $u$ is a morphism outside of a locus of codimension two.
As such a locus does not dominate $Y$,
up to further shrinking $Y$, 
we can assume that $u$ is a morphism.

As $u$ is separable, by generic smoothness~\cite[\href{https://stacks.math.columbia.edu/tag/056V}{Tag 056V}]{stacks-project}
the tangent map $du : T_{Y \times \P^1} \to u^*T_X$
is surjective at a general point $(y,t) \in Y \times \P^1$.
Write $\P^1_y \cnec \{y\} \times \P^1$, then
$(du)|_{\P_y^1}$ is generically surjective.
Thus
we have 
a morphism of sheaves on $\P^1$
\[
\cO_{\P_y^1}(2) \oplus \cO_{\P_y^1}^{\dim Y} \simeq
T_{Y \times \P^1}|_{\P^1_y}  \longrightarrow 
(u^*T_X)|_{\P^1_y} \simeq
\bigoplus_{i = 1}^{\dim X} \cO(a_i)
\]
with torsion cokernel
and it follows immediately
that
all 
$a_i \ge 0$, so 
$(u^*T_X)|_{\P^1_y}$ 
is globally generated 
and $u|_{\P^1_y} : \P^1_y \to X$ is a free morphism.
\end{proof}

Recall that a variety $W$ is called separably rationally connected~\cite[IV.3.2.3]{Kollarrat}
if there exist a variety $S$ and a rational map
$u : S \times \P^1 \dto W$ such that the two-point evaluation map
\begin{equation}\label{eq:rat-conn}
    u_2 = u \times_S u : S \times \P^1 \times \P^1  \dto W \times W
\end{equation}
is dominant and separable.
Note that this condition implies that $u$ is dominant 
and separable.

\begin{lem}\label{lem-stbirCY} 
Let $Z$ and $Z'$ be smooth projective varieties 
of the same dimension
over 
an algebraically closed field $\k$.
Let $W$ be a separably rationally connected variety over $\k$.
Suppose that both $Z$ and $Z'$ are not separably uniruled.
Then for
every birational map
$\phi : Z \times W \dto Z' \times W$,
there is a unique birational map
$\ol{\phi} : Z \dto Z'$ making the diagram
\begin{equation}\label{eq-ZZW}
    \xymatrix{
Z \times W \ar[d] \ar@{-->}[r]^{\phi} & Z' \times W \ar[d] \\
Z  \ar@{-->}[r]^{\ol{\phi}} & Z'  \\
}
\end{equation}
commutative \textup(here, the vertical arrows are the
projections onto the first factors\textup).
\end{lem}

See~\cite[Theorem 2]{LiuSebag} for a similar result,
where 
a more general statement 
was proven in characteristic zero by Liu and Sebag.
Our proof follows a similar strategy
to that of \cite[Theorem 2]{LiuSebag},
relying on the statement analogous
to \cite[Lemma 5]{LiuSebag}.

\begin{proof}

Let 
$$p: Z \times W \dto Z' \times W \to  Z' $$
be the composition of $\phi$ with the projection.
Since $W$ is separably rationally connected,
there exist a variety $S$ and a rational map
$u : S \times \P^1 \dto W$ such that
the map $u_2$ \eqref{eq:rat-conn}
is dominant and separable.
Since $u$ is also dominant and separable,
so is the composition
$$\Phi: Z \times S \times \P^1 
\overset{\Id_Z \times u}{\dto} Z  \times W
\overset{p}{\dto} Z'.$$

Since $Z'$ is not separably uniruled, $\Phi$ contracts
$\{(z, s)\} \times \P^1$ for general $(z, s) \in Z \times S$.
In other words, $p(z,-): W \dto Z'$ 
contracts general rational curves of $W$ parameterized
by $S$ through $u$. 
Since these curves connect general points of $W$
because $u_2$ is dominant,
$p(z,-)$ contracts $W$ for general $z \in Z$. 

Let $\Gamma \subset (Z \times W) \times Z'$ be the closure of the graph
of $p$ and let $\ol{\Gamma} \subset Z \times Z'$ 
be the image of $\Gamma$.
The fiber of $\ol{\Gamma} \to Z$ over a general point 
$z \in Z$ is the point
$p(\{z\} \times W)$. 
Moreover, $\ol{\Gamma} \to Z$ has degree one,
because the composition of projections
$$(Z \times W) \times Z' \to Z \times Z' \to Z$$
induces the composition of surjective morphisms
$$\Gamma_w \cnec \Gamma \cap (Z \times \{w\} \times Z')
\to \ol{\Gamma} \to Z,$$
which is birational for general $w \in W$;
this is because $\Gamma_w$ is the graph of the rational
map $p|_{Z \times \{w\}}$.
Therefore $\ol{\Gamma}$
is the graph of a rational dominant map $\ol{\phi}: Z \dto Z'$,
and it makes~\eqref{eq-ZZW} commutative by construction.
It is clear from~\eqref{eq-ZZW} that $\phi$ uniquely determines $\ol{\phi}$.

To show that $\ol{\phi}$ is birational,
we apply the same construction
to $\phi^{-1}$, and by uniqueness
both compositions of
$\ol{\phi}$ with
$\ol{\phi^{-1}}$ are identity maps.
\end{proof}

\begin{cor}\label{cor-CYnSB} 
Let $Z$ and $Z'$ be K-trivial varieties 
of Picard number $1$
over 
an algebraically closed field $\k$.
If $Z \not\simeq Z'$,
then $Z \times W$ is not birational to $Z' \times W$
for any separably rationally connected variety $W$ over $\k$.
In particular, $Z$ is not stably birational to $Z'$.
\end{cor}

\begin{proof} 

Since K-trivial varieties are not separably uniruled~\cite[Corollary IV.1.11]{Kollarrat}, 
Theorem~\ref{thm:BurtCY}
and Lemma~\ref{lem-stbirCY} show that
$Z \times W$ is not birational to $Z' \times W$.
\end{proof}

\section*{Acknowledgments} 

We would like to thank 
J.\,Blanc,
P.\,Cascini,
F.\,Gounelas,
D.\,Huybrechts,
A.\,Kuznetsov,
\mbox{K.-W.\,Lai,}
S.\,Lamy,
K.\,Oguiso,
Z.\,Patakfalvi,
Yu.\,Prokhorov,
J.\,Schneider,
A. Shnidman~\cite{MO-Shnidman},
C.\,Shramov,
Y.\,Toda,
S.\,Zimmermann 
for their advice and interest in our work.
We thank B.\,Totaro for offering us a proof of Theorem~\ref{thm:BurtCY} which generalizes and replaces a  lemma from a previous draft.
We thank the referee for their comments and suggestions for
improving the paper.
This project was conceived when the second-named
author visited Osaka and Tokyo
in 2020, and we thank Yu.\,Kawamata,
K.\,Oguiso and S.\,Okawa for organizing and supporting this visit.

HYL is supported by
	Taiwan Ministry of Education Yushan Young Scholar Fellowship (NTU-110VV006),
	and Taiwan Ministry of Science and Technology (110-2628-M-002-006-).
	ES is supported by the EPSRC grant
    EP/T019379/1 ``Derived categories and algebraic K-theory of singularities'', and by the
    ERC Synergy grant ``Modern Aspects of Geometry: Categories, Cycles and Cohomology of Hyperkähler Varieties".

\bibliographystyle{plain}
\bibliography{cremona-paper}

\end{document}